\documentclass[reqno,dvips,hdvips]{amsart}

\usepackage{amsmath, amssymb, url}
\usepackage{xcolor}

\newtheorem{theorem}                 {Theorem}      [section]
\newtheorem{proposition}  [theorem]  {Proposition}

\newtheorem{lemma}        [theorem]  {Lemma}
\newtheorem{remark}        [theorem]  {Remark}
\newtheorem{definition}  [theorem]  {Definition}

\def\p{{\mathcal{P}}}
\def\D{{\mathcal{D}}}
\def\Dp{{\mathcal{D}}^\perp}
\def\Ni{{N_\top}}
\def\Na{N_\perp}
\def\j{{\rm j}}

\def\({\big(}
\def\){\big)}
\def\<{\langle}
\def\>{\rangle}

\def\dsi{\partial_{s_i}}
\def\dsj{\partial_{s_j}}
\def\dti{\partial_{t_i}}
\def\dtj{\partial_{t_j}}

\def\dua{\partial_{u_a}}
\def\dub{\partial_{u_b}}

\title{Geometry of ${\mathcal{P}} R$-warped products in para-K\"ahler manifolds}
\author[B.-Y. Chen]{Bang-Yen Chen}
\author[M.~I. Munteanu]{Marian Ioan Munteanu}

\address[B.-Y.~Chen and M.~I.~Munteanu]
{Michigan State University\\
Department of Mathematics\\
Wells Hall\\
48824-1029 East Lansing\\ USA}
\email{bychen (at) math.msu.edu}
\email{marian.ioan.munteanu (at) gmail.com}

\address[M.~I.~Munteanu]
{Al.I. Cuza University of Iasi\\
Faculty of Mathematics\\
Bd. Carol I, no. 11\\
700506 - Iasi\\
Romania}

\date{\today}



\begin{document}

\begin{abstract} In this paper, we initiate the study of $\p R$-warped products in para-K\"ahler manifolds and prove some fundamental results on such submanifolds.  In particular, we establish
a general optimal inequality for $\p R$-warped products in para-K\"ahler manifolds involving only the warping function and the second fundamental form. Moreover, we completely classify $\p R$-warped
products in the flat para-K\"ahler manifold with least codimension which satisfy the equality case of the inequality.

\end{abstract}
\maketitle

\numberwithin{equation}{section}


\section{Introduction}

An almost para-Hermitian manifold is a manifold $\widetilde M$ equipped with an almost product structure
$\p\ne \pm I$ and a pseudo-Riemannian metric $\widetilde g$ such that
\begin{align}
\label{1.1}\p^2=I,\;\; \widetilde g(\p X,\p Y)=-\widetilde g(X,Y),
\end{align}
for vector fields $X$, $Y$ tangent to $\widetilde M$, where $I$ is the identity map. Clearly, it follows from \eqref{1.1} that
the dimension of $\widetilde M$ is even and the metric $\widetilde g$ is neutral.
  An almost para-Hermitian manifold  is called  {\it para-K\"ahler}  if it satisfies $\widetilde \nabla \p=0$ identically, where  $\widetilde \nabla$ denotes the Levi Civita connection of $\widetilde M$. We define $||X ||_{2}$ associated with $\widetilde g$ on $\widetilde M$  by $||X||_2=\widetilde g(X,X)$.

 Properties of para-K\"ahler manifolds were first studied  in 1948 by  Rashevski who considered a neutral metric
 of signature $(m,m)$ defined from a potential function on a locally product $2m$-manifold  \cite{rash48}.
 He called such manifolds stratified spaces. Para-K\"ahler manifolds were explicitly defined by  Rozenfeld in 1949 \cite{Roz}. Such manifolds were also defined by Ruse in 1949 \cite{ruse} and studied by  Libermann  \cite{Li} in the context of $G$-structures.

  There exist many para-K\"ahler manifolds, for instance,  it was proved in \cite{H1} that a homogeneous manifold  $\widetilde M=G/H$ of a semisimple Lie group $G$  admits an invariant para-K\"ahler structure  $(\widetilde g,\p)$ if and only if it is a covering of the adjoint orbit  ${\rm Ad}_G h$ of a semisimple element $h$.  (For a very nice survey on para-K\"ahler manifolds, see \cite{E}.)
  Para-K\"ahler manifolds have been applied in supersymmetric field theories as well as in string theory in recent years (see, for instance,
  \cite{cortes,cortes1,cortes2}).

 A pseudo-Riemannian submanifold $M$ of a para-K\"ahler manifold $\widetilde M$ is called {\it invariant} if the tangent bundle of $M$ is invariant under the action of $\p$. $M$ is called {\it anti-invariant}  if $\p$ maps each tangent space $T_pM, \, p\in M,$ into the normal space $T_p^\perp M$. A {\it Lagrangian submanifold} $M$ of a para-K\"ahler manifold $\widetilde M$ is an anti-invariant submanifold satisfying $\dim \widetilde M=2\dim M$. Such submanifolds have been investigated recently in  \cite{nl,c2011,tjm,cm:Chen11}.

 A pseudo-Riemannian submanifold $M$ of a para-K\"ahler manifold $\widetilde M$ is called a
\emph{$\p R$-submanifold} if the tangent bundle $TM$ of $M$ is the direct sum of  
an \emph{invariant} distribution $\D$ and an
\emph{anti-invariant} distribution $\Dp$, i.e.,  $$T(M)=\D\oplus\Dp, \;\; \p\D=\D,\;\;  \p\Dp\subseteq T_p^\perp(M).$$
 
A $\p R$-submanifold is called a \emph{$\p R$-warped product} if it is a warped product $\Ni \times_f \Na$ of an invariant submanifold $\Ni$ and an anti-invariant submanifold $\Na$. 

In this paper we initiate  the  study  of $\p R$-warped products in para-K\"ahler manifolds. 
The basic properties of $\p R$-warped products are given in section 3. 
We establish in section 4 a general optimal inequality for $\p R$-warped products in para-K\"ahler manifolds involving only the warping function and
the second fundamental form.  In section 5, we provide the exact solutions of a PDE system associated with $\p R$-warped products. In the last section, we  classify  $\p R$-warped products $\Ni \times_f \Na$  with least codimension in the flat para-K\"ahler
manifold which verify the equality case of the general  inequality derived in section 4.


\section{Preliminaries}

\subsection{Warped product manifolds}

The notion of warped product (or, more generally warped bundle)
was introduced by Bishop and O'Neill in \cite{cm:BON69} in order to
construct a large variety of manifolds of negative curvature. For
example, negative space forms can easily be constructed in this way
from flat space forms. The interest of geometers was to extend the
classical de Rham theorem to warped products. Hiepko proved a result
in \cite{cm:Hie79} which will be used in this paper.

Let us recall some basic results on warped products.
Let $B$ and $F$ be two pseudo-Riemannian manifolds with pseudo-Riemannian metrics
$g_B$ and $g_F$ respectively, and $f$ a positive function on $B$. Consider the product manifold $B\times F$. Let $\pi_1:B\times F\longrightarrow B$ and $\pi_2:B\times F\longrightarrow F$
be the canonical projections.

We define the manifold $M=B\times_f F$ and call it {\it warped
product} if it is equipped with the following warped metric
\begin{equation}
\label{formula2.4}
g(X,Y)=g_B\big(\pi_{1_{*}}\!(X),\pi_{1_{*}}\!(Y)\big)+f^2(\pi_1(p))g_F\big(\pi_{2_{*}}\!(X),\pi_{2_{*}}\!(Y)\big)
\end{equation}
for all $X,Y\in T_p(M)$, $p\in M$, or equivalently,
\begin{equation}
\label{formula2.5} g=g_B+f^2\ g_F.
\end{equation}
The function $f$ is called {\it the warping function}.
For the sake of simplicity we will identify a vector field $X$ on $B$ (respectively, a vector field $Z$ on $F$)
with its lift $\tilde X$ (respectively $\tilde Z$) on $B\times_fF$.

If $\nabla$, $\nabla^B$ and $\nabla^F$ denote the Levi-Civita connections of $M$, $B$ and $F$, respectively,
then the following formulas hold
\begin{equation}
\label{warped_conn_eq}
\begin{array}{l}
\nabla_XY=\nabla^B_XY,\\
\nabla_XZ=\nabla_ZX=X(\ln f)~Z,\\
\nabla_ZW=\nabla^F_ZW-g(Z,W)~\nabla(\ln f)
\end{array}
\end{equation}
where $X,Y$ are tangent to $B$ and $Z,W$ are tangent to $F$. Moreover, $\nabla(\ln f)$ is the gradient
of $\ln f$ with respect to the metric $g$.


\subsection{Geometry of submanifolds}
Let $M$ be an $n$-dimensional submanifold of $\widetilde{M}$.
We need the Gauss and Weingarten formulas:
$$
{\mathbf{(G)}}\quad
    \widetilde{\nabla}_X Y = \nabla_X Y + \sigma(X,Y) , \qquad
{\mathbf{(W)}}\quad  \widetilde{\nabla}_X \xi = -A_\xi X + \nabla^\perp_X \xi \, ,
$$
for vector fields $X,Y$ tangent to $M$ and $\xi$ normal to $M$, where  $\nabla$ is the induced connection, $\nabla^\perp$ is the normal connection on the normal bundle $T^\perp(M)$,
$\sigma$ is the second fundamental form, and  $A_\xi$ is the
shape operator associated with the normal section $\xi$. The mean curvature vector $H$ of $M$ is defined by $H=\frac{1}{n}{\rm trace h}$.

For later use we recall the equations of Gauss and Codazzi:
$$
\begin{array}{l}
{\mathbf{(EG)}}\quad
     g\(R_{XY}Z,W\)=\widetilde g\(\widetilde R_{XY}Z,W\)+\widetilde g\(\sigma(Y,Z),\sigma(X,W)\)-
                                                         \widetilde g\(\sigma(X,Z),\sigma(Y,W)\),
\\[2mm]
{\mathbf{(EC)}}\quad
(\widetilde R_{XY}Z)^\perp=(\bar\nabla_X\sigma)(Y,Z)-(\bar\nabla_Y\sigma)(X,Z)
\end{array}
$$
for $X,Y,Z$ and $W$ tangent to $M$, where $R$, $\widetilde R$ are the curvature tensors on $M$ and $\widetilde M$,
respectively, $(\widetilde R_{XY}Z)^\perp$ is the normal component of $\widetilde R_{XY}Z$ and $\bar\nabla$ is the
van der Waerden - Bortolotti connection defined as
\begin{equation}
\label{WBconn:eq}
(\bar\nabla_X\sigma)(Y,Z)=\nabla^\perp_X\sigma(Y,Z)-\sigma(\nabla_XY,Z)-\sigma(Y,\nabla_X,Z).
\end{equation}
In this paper the curvature is defined by $R_{XY}=[\nabla_X,\nabla_Y]-\nabla_{[X,Y]}$.

A submanifold is called {\em totally geodesic} if its second fundamental form vanishes identically.
For a normal vector field $\xi$ on $M$, if $A_\xi=\lambda \,
I$, for certain function $\lambda$
on $M$, then $\xi$ is called a {\em umbilical section} (or $M$ is  {\em umbilical with respect to} $\xi$).
If  $M$ is umbilical with respect to every (local) normal vector field, then $M$
is called a {\em totally umbilical submanifold}. A pseudo-Riemannian submanifold is called {\it minimal} if the mean curvature vector $H$ vanishes identically. And it is called
 {\it quasi-minimal} if $H$ is a light-like vector field.

Recall that for a warped product $M=B\times_f F$,  $B$ is totally geodesic and $F$ is totally umbilical in $M$.

\subsection{Para-K\"ahler $n$-plane}

The simplest example of para-K\"ahler manifold is the para-K\"ahler $n$-plane $({\mathbb E}^{2n}_n,{\mathcal P},g_0)$ consisting of the pseudo-Euclidean $2n$-space $\mathbb E^{2n}_n$, the standard flat neutral metric
\begin{align}\label{10.2} g_0=-\text{$\sum$}_{j=1}^n dx_j^2+\text{$\sum$}_{j=1}^n dy_j^2,\end{align} and  the almost product structure  \begin{align}{\mathcal P}=\text{$\sum$}_{j=1}^n \text{\small$\frac{\partial}{\partial y_j}$}\otimes dx_j+\text{$\sum$}_{j=1}^n
\text{\small$\frac{\partial}{\partial x_j}$} \otimes dy_j.\end{align}
We simply denote the para-K\"ahler $n$-plane $({\mathbb E}^{2n}_n,{\mathcal P},g_0)$  by $\p^{n}$. 


\section{$\p R$-submanifolds of para-K\"ahler manifolds}

For any vector field $X$ tangent to $M$, we put $P X = tan(\p X)$ and $F X = nor (\p X)$, where $tan_p$ and $nor_p$ are the natural projections
associated to the direct sum decomposition
$$
T_p (\widetilde{M}) = T_p (M) \oplus T_p^\perp(M) \ , \ p \in M.
$$
Then $P$ is an endomorphism of the tangent bundle $T(M)$
and $F$ is a normal bundle valued 1-form on $M$.
Similarly, for a normal vector field $\xi$, we put
$t\xi=tan(\p \xi)$ and $f\xi=nor(\p \xi)$ for
the tangential and the normal part of $\p \xi$, respectively.

Let $\nu$ denote the orthogonal complement of $\p\Dp$ in $T^\perp(M)$. Then we have
$$
   T^\perp(M)=\p\Dp\oplus\nu.
$$
Notice that $\nu$ is invariant, i.e., $\p\nu=\nu$.

The following proposition characterizes $\p R$-submanifold of para-K\"ahler manifolds.
A similar result is known for $CR$-submanifolds in K\"ahlerian manifolds and 
contact $CR$-submanifolds in Sasakian manifolds. See e.g. \cite{cm:YK83}.

\begin{proposition}
Let $M\to\widetilde M$ be an isometric immersion of a pseudo-Riemannian manifold $M$ into a para-K\"ahler manifold $\widetilde M$. Then a necessary and sufficient condition for $M$ to be a $\p R$-submanifold is that $F\circ P=0$.
\end{proposition}
\proof
For $U$ tangent to $M$ we have the following decomposition
$$
   U=\p^2 U=P^2U+FPU+tFU+fFU.
$$
By identifying the tangent and the normal parts respectively, we find
$$
   P^2+tF=I \quad {\rm and} \quad FP+fF=0.
$$
Suppose that $M$ is a $\p R$-submanifold. After we choose $U=X\in\D$ we have $\p X=PX$ and $FX=0$.
Hence $P^2=I$ and $FP=0$ on $\D$. On the other hand, if $U=Z=\Dp$, we have $PZ=0$. Hence $FP=0$ on $\Dp$ too.

Conversely, suppose that $FP=0$. Put
$$
    \D=\{X\in T(M) : \p X\in T(M)\}\ {\rm and}\
    \Dp=\{Z\in T(M) : \p Z\in T^\perp(M)\}.
$$
Then by direct computations we conclude that $\D$ and $\Dp$ are orthogonal such that $T(M)=\D\oplus\Dp$.
\endproof
The following results from \cite{cm:Chen11} are necessary for our further computations.

\begin{proposition}
Let $M$ be a $\p R$-submanifold of a para-K\"ahler manifold $\widetilde M$. Then
\begin{itemize}
\item[(i)] the anti-invariant distribution $\Dp$ is a non-degenerate integrable distribution;
\item[(ii)] the invariant distribution $\D$ is a non-degenerate minimal distribution;
\item[(iii)] the invariant distribution $\D$ is integrable if and only if $\sigma(PX,Y)=\sigma(X,PY)$,
for all $X,Y\in\D$;
\item[(iv)]  $\D$ is integrable if and only if $\dot\sigma$ is symmetric,
equivalently to $\dot\sigma(PX,Y)=\dot\sigma(X,PY)$. Here $\dot\sigma$ denotes the second fundamental
form of $\D$ in $M$.
\end{itemize}
\end{proposition}

Now, let us give some useful formulas.

\begin{lemma}
\label{lemma_3.3}
If $M$ is a $\p R$-submanifold of a para-K\"ahler manifold $\widetilde M$,
 then
 \begin{itemize}
 \item[(a)] $\widetilde g(A_{FZ}U,PX)=g(\nabla_UZ,X)$,
 \item[(b)] $A_{FZ}W=A_{FW}Z$ and $A_{f\xi}X=-A_\xi PX$,
 \end{itemize}
 for all $X,Y\in\D$, $Z,W\in\Dp$, $U\in T(M)$ and $\xi\in\Gamma(\nu)$.
\end{lemma}

We need the following for later use.

\begin{proposition}
Let $M$ be a $\p R$-submanifold of a para-K\"ahler manifold $\widetilde M$. Then
\begin{itemize}
 \item[(i)] the distribution $\Dp$ is totally geodesic if and only if
\begin{equation}
\label{Dp_tg:eq}
\widetilde g(\sigma(\D,\Dp),\p\Dp)=0
\end{equation}
\item[(ii)] the distribution $\D$ is totally geodesic if and only if
\begin{equation}
\label{D_tg:eq}
\widetilde g(\sigma(\D,\D),\p\Dp)=0
\end{equation}
\item[(iii)]  $\D$ is totally umbilical if and only if there exists $Z_0\in\Dp$such that
\begin{equation}
\label{Dp_tu:eq}
             \sigma(X,Y)=g(X,PY)~FZ_0\ ({\rm mod}~\nu)\ ,\ \forall~X,Y\in\D.
\end{equation}
\end{itemize}
\end{proposition}
\proof This can be proved by
classical computations: see e.g. \cite{cm:Chen81} or \cite{cm:Mun05}.
\endproof

\subsection{$\p R$-products}
A $\p R$-submanifold of a para-K\"ahler manifold  is called a \emph{$\p R$-product}
if it is locally a direct product $\Ni\times \Na$ of an invariant submanifold $\Ni$ and an anti-invariant
submanifold $\Na$.

The next result characterizes $\p R$-products in terms of the operator $P$.

\begin{proposition}[Characterization]
A $\p R$-submanifold of a para-K\"ahler manifold is a $\p R$-product if and only if $P$ is parallel.
\end{proposition}
\proof
By straightforward computations (as in \cite[Theorem 4.1]{cm:Chen81} or \cite[Theorem 2.2]{cm:Mun05})
we may prove that
$$(\nabla_UP)V=\nabla_U(PV)-P\nabla_UV=0\ ,\ \forall\; U,V\in\chi(M),
$$ which implies the desired result.
\endproof

The following result was proved in \cite[page 224]{cm:Chen11}.

\begin{proposition} \label{PR} Let $N_{\top}\times N_{\perp}$ be a ${\mathcal P}R$-product of the para-K\"ahler $(h+p)$-plane $\p^{h+p}$ with $h=\frac{1}{2}\dim N_{\top}$ and $p=\dim N_{\perp}$.
If $N_{\perp}$ is either spacelike or timelike, then the ${\mathcal P}R$-product  is an open part of a direct product of a para-K\"ahler $h$-plane $\p^{h}$ and a Lagrangian submanifold $L$ of $\p^{p}$, i.e., 
$$ \Ni\times \Na\subset \p^{h}\times L\subset \p^{h}\times \p^{p}=\p^{h+p}.$$
\end{proposition}


\subsection{$\p R$-warped products}

Let us begin with the following result.
\begin{proposition}
\label{non_exist_th} If a $\p R$-submanifold $M$ is  a warped product $\Na\times_f\Ni$ of an anti-invariant submanifold
$\Na$ and an invariant submanifold $\Ni$ with  warping function $f:\Na\longrightarrow{\mathbb{R}}_+$,
then $M$ is a $\p$R product $\Na\times N_\top^f$, where $N_\top^f$ is the manifold $\Ni$ endowed with the homothetic
metric $g_\top^f=f^2g_\top$.
\end{proposition}
\proof
Consider $X,Y\in\D$ and $Z\in\Dp$. Compute
$$
\begin{array}{l}
\widetilde g(\sigma(X,Y),FZ)=\widetilde g(\widetilde\nabla_XY,\p Z)=-\widetilde g(Y,\p\widetilde\nabla_XZ)= g(PY,\nabla_XZ)=\\
\qquad = g(PY,Z(\ln f)~X)=Z(\ln f)~g(X,PY).
\end{array}
$$
Since $\sigma(\cdot ~,\, \cdot)$ is symmetric and $g(\cdot~,\, P\, \cdot)$ is skew-symmetric, it follows that $Z(\ln f)$ vanishes for all $Z$ tangent to $\Na$.
Consequently, $f$ is a constant and thus the warped product is nothing but the product $\Na\times N_\top^f$.
\endproof

The previous result shows that there do not exist warped product $\p R$-submanifolds  in para-K\"aehler manifolds of the form
$\Na\times_f\Ni$, other than $\p R$-products. Thus, in view of Proposition \ref{PR} we
give the following definition:

\begin{definition} {\rm A $\p R$-submanifold of a para-K\"ahler manifold $\widetilde M$ is called a \emph{$\p R$-warped product} if it is a warped product of the form: $\Ni \times_f \Na$,
where $\Ni$ in an invariant submanifold, $\Na$ is an anti-invariant submanifold of $\widetilde M$ and $f$ is a non-constant  function $f:N_\top\to {\mathbb R}_+$.}
\end{definition}

Since the metric on $N_{T}$ of a $\p R$-warped product $\Ni\times_f N_\perp$ is neutral, we simply called  the $\p R$-warped product $\Ni\times_f N_\perp$ {\it space-like} or {\it time-like} depending on $N_\perp$ is space-like
or time-like, respectively.

The next result characterizes $\p R$-warped products in para-K\"ahler manifolds.

\begin{proposition}
Let $M$ be a proper $\p R$-submanifold of a para-K\"ahler manifold. Then $M$ is a $\p R$-warped product if and only if
\begin{equation}
\label{warped_prod_eq}
A_{FZ}X=(PX(\mu))\,Z\ ,\ \forall \ X\in\D, \ Z\in\Dp,
\end{equation}
for some smooth function $\mu$ on $M$ satisfying $W(\mu)=0$, $\forall~W\in\Dp$.
\end{proposition}

The proof of this result is similar as in the case of K\"ahler or Sasakian ambient space. The key  is the characterization of warped products given by Hiepko in \cite{cm:Hie79}.

\section{An optimal inequality}

\begin{theorem}
\label{ineq:th}
Let $M=\Ni\times_{f}\Na$ be a $\p R$-warped product in a para-K\"ahler manifold $\widetilde M$.
Suppose that $\Na$ is space-like and $\nabla^{\perp}\(\p\Na)\subseteq\p\Na$.
Then the second fundamental form of $M$ satisfies
\begin{equation}
\label{ineq:eq}
S_\sigma \leq 2p {||\nabla\ln f||}_2+{||\sigma_\nu^{\D}||}_2,
\end{equation}
where $p=\dim\Na$, $S_\sigma=\widetilde g(\sigma,\sigma)$, $\nabla\ln f$ is the gradient of $\ln f$ with respect to the metric $g$ 
and ${||\sigma_\nu^{\D}||}_2=\widetilde g\(\sigma_\nu(\D,\D),\sigma_\nu(\D,\D)\)$.
Here the index $\nu$ represents the $\nu$-component of that object.
\end{theorem}
\proof
If we denote by $g{}_\top$ and $g{}_\perp$ the metrics on $\Ni$ and $\Na$, then the warped metric on $M$ is
$g=g_\top+f^2g_\perp$.
Let us consider
\begin{itemize}
\item on $\Ni$: an orthonormal basis $\{X_i,X_{i*}=PX_i\}$, $i=1,\ldots,h$, where\linebreak $h=\dim\Ni$;
     moreover, one can suppose that $\epsilon_i:=g(X_i,X_i)=1$ and hence $\epsilon_{i*}:=g(X_{i*},X_{i*})=-1$, for all $i$.
\item on $\Na$: an orthonormal basis $\{\tilde Z_a\}, a=1,\ldots,p$; we put $\epsilon_a:=g_\perp(\tilde Z_a,\tilde Z_a)=1$,
    for all $a$;
\item in each point $(x,y)\in M$: $Z_a(x,y)=\frac 1{f(x)}~\tilde Z_a(y)$;
\item in $\nu$: an orthonormal basis $\{\xi_\alpha,\xi_{\alpha*}=f\xi_{\alpha*}\}$, $\alpha=1,\ldots,q$;
      moreover, one can suppose that $\epsilon_\alpha:=\widetilde g(\xi_\alpha,\xi_\alpha)=1$ and hence
      $\epsilon_{\alpha*}:=\widetilde g(\xi_{\alpha*},\xi_{\alpha*})=-1$.
\end{itemize}

Now, we want to compute $\widetilde g(\sigma,\sigma)=$
$$=\widetilde g\(\sigma(\D,D),\sigma(\D,\D)\)+
                            2\widetilde g\(\sigma(\D,\Dp),\sigma(\D,\Dp)\)+
                            \widetilde g\(\sigma(\Dp,\Dp),\sigma(\Dp,\Dp)\),
$$
where
\begin{equation}
\label{sigmadd:eq}
\begin{array}{c}
\widetilde g\(\sigma(\D,\D),\sigma(\D,\D)\)=\sum\limits_{i,j=1}^h
    \Big(\epsilon_i\epsilon_j\widetilde g\(\sigma(X_i,X_j),\sigma(X_i,X_j)\)\\[2mm]
   + \epsilon_{i*}\epsilon_j\widetilde g\(\sigma(X_{i*},X_j),\sigma(X_{i*},X_j)\)
   + \epsilon_i\epsilon_{j*}\widetilde g\(\sigma(X_{i},X_{j*}),\sigma(X_{i},X_{j*})\)\\[2mm]
   + \epsilon_{i*}\epsilon_{j*}\widetilde g\(\sigma(X_{i*},X_{j*}),\sigma(X_{i*},X_{j*})\)
                        \Big),
\end{array}
\end{equation}

\begin{equation}
\label{sigmaddp:eq}
\begin{array}{c}
\widetilde g\(\sigma(\D,\Dp),\sigma(\D,\Dp)\)=\sum\limits_{i=1}^h\sum\limits_{a=1}^p\Big(
    \epsilon_i\epsilon_a\widetilde g\(\sigma(X_i,Z_a),\sigma(X_i,Z_a)\)\\[2mm]
  +  \epsilon_{i*}\epsilon_a\widetilde g\(\sigma(X_{i*},Z_a),\sigma(X_{i*},Z_a)\)\Big)
\end{array}
\end{equation}
and
\begin{equation}
\label{sigmadpdp:eq}
\widetilde g\(\sigma(\Dp,\Dp),\sigma(\Dp,\Dp)\)=\sum\limits_{a,b=1}^p\epsilon_a\epsilon_b\widetilde g\(\sigma(Z_a,Z_b),\sigma(Z_a,Z_b)\).
\end{equation}

To do so, first we analyze $\sigma(\D,\D)$.
Since $\D$ is totally geodesic, we have $\sigma(\D,\D)\in\nu$.
Hence one can write the following
$$
\begin{array}{rclcrcl}
\sigma(X_i,X_j)&=&\sigma_{ij}^\alpha\xi_\alpha+\sigma_{ij}^{\alpha*}\xi_{\alpha*} , &&
\sigma(X_{i*},X_j)&=&\sigma_{i*j}^\alpha\xi_\alpha+\sigma_{i*j}^{\alpha*}\xi_{\alpha*}, \\[2mm]
\sigma(X_{i*},X_{j*})&=&\sigma_{i*j*}^\alpha\xi_\alpha+\sigma_{i*j*}^{\alpha*}\xi_{\alpha*}, &&
\sigma(X_{i},X_{j*})&=&\sigma_{ij*}^\alpha\xi_\alpha+\sigma_{ij*}^{\alpha*}\xi_{\alpha*}.
\end{array}
$$
It follows that
\begin{equation}
\label{sdd:eq}
\begin{array}{c}
\widetilde g\(\sigma(\D,\D),\sigma(\D,\D)\) = \sum\limits_{i,j=1}^h\sum\limits_{\alpha=1}^q\Big\{
\big[(\sigma_{ij}^\alpha)^2-(\sigma_{ij}^{\alpha*})^2\big]-\big[(\sigma_{i*j}^\alpha)^2-(\sigma_{i*j}^{\alpha*})^2\big]\\[2mm]
 -\big[(\sigma_{ij*}^\alpha)^2-(\sigma_{ij*}^{\alpha*})^2\big]+\big[(\sigma_{i*j*}^\alpha)^2-(\sigma_{i*j*}^{\alpha*})^2\big]
\Big\}.
\end{array}
\end{equation}
Due to the integrability of $\D$ we deduce that $\sigma_{i*j}^\alpha=\sigma_{ij*}^\alpha$, $\sigma_{i*j}^{\alpha*}=\sigma_{ij*}^{\alpha*}$,
$\sigma_{i*j*}^\alpha=\sigma_{ij}^\alpha$, $\sigma_{i*j*}^{\alpha*}=\sigma_{ij}^{\alpha*}$.
Furthermore, using Lemma~\ref{lemma_3.3}, we may write
$$
    \widetilde g\(\sigma(X,Y),\xi\)=-\widetilde g\(\sigma(X,PY),f\xi\)\ ,\ \forall~X,Y\in\D, \ \xi\in\nu
$$
and consequently we have
$$
\begin{array}{l}
\sigma_{ij}^\alpha=\widetilde g\(\sigma(X_i,X_j),\xi_\alpha\)=
        -\widetilde g\(\sigma(X_i,X_{j*}),\xi_{\alpha*}\)=\sigma_{ij*}^{\alpha*}, \\[1mm]
\sigma_{ij}^{\alpha*}=-\widetilde g\(\sigma(X_i,X_j),\xi_{\alpha*}\)
        =\widetilde g\(\sigma(X_i,X_{j*}),\xi_{\alpha}\)=\sigma_{ij*}^{\alpha}.
\end{array}
$$
By replacing all these in \eqref{sdd:eq}, we obtain
\begin{equation}
\label{sdd0}
\widetilde g\(\sigma(\D,\D),\sigma(\D,\D)\)={||\sigma^\D_\nu||}_2=
    4\sum\limits_{i,j=1}^h\sum\limits_{\alpha=1}^q
    \big[(\sigma_{ij}^\alpha)^2-(\sigma_{ij}^{\alpha*})^2\big].
\end{equation}

Let us focus now on $\widetilde g\(\sigma(\D,\Dp),\sigma(\D,\Dp)\)$.
As before, we write
$$
\begin{array}{rcl}
\sigma(X_i,Z_a)&=&\sigma_{ia}^bFZ_b+\sigma_{ia}^\alpha\xi_\alpha+\sigma_{ia}^{\alpha*}\xi_{\alpha*} ,\\[2mm]
\sigma(X_{i*},Z_a)&=&\sigma_{i*a}^bFZ_b+\sigma_{i*a}^\alpha\xi_\alpha+\sigma_{i*a}^{\alpha*}\xi_{\alpha*}.
\end{array}
$$
It follows that
$$
\begin{array}{rcl}
\widetilde g\(\sigma(X_i,Z_a),\sigma(X_i,Z_a)\)&=&- \sum\limits_{b=1}^p(\sigma_{ia}^b)^2+
        \sum\limits_{\alpha=1}^q\big[(\sigma_{ia}^\alpha)^2-(\sigma_{ia}^{\alpha*})^2\big], \\[2mm]
\widetilde g\(\sigma(X_{i*},Z_a),\sigma(X_{i*},Z_a)\)&=&- \sum\limits_{b=1}^p(\sigma_{i*a}^b)^2+
        \sum\limits_{\alpha=1}^q\big[(\sigma_{i*a}^\alpha)^2-(\sigma_{i*a}^{\alpha*})^2\big].
\end{array}
$$
We obtain
\begin{equation}
\label{ddp_eq}
\begin{array}{l}
\widetilde g\(\sigma(\D,\Dp),\sigma(\D,\Dp)\)=-\sum\limits_{i=1}^h\sum\limits_{a,b=1}^p
                \big[(\sigma_{ia}^b)^2-(\sigma_{i*a}^b)^2\big]\qquad\qquad\\[2mm]
   \qquad \qquad  +\sum\limits_{i=1}^h\sum\limits_{a=1}^p\sum\limits_{\alpha=1}^q
                \big[(\sigma_{ia}^\alpha)^2-(\sigma_{ia}^{\alpha*})^2-(\sigma_{i*a}^\alpha)^2+(\sigma_{i*}^{\alpha*})^2\big].
\end{array}
\end{equation}
From Lemma~\ref{lemma_3.3} we have
$$
\widetilde g\(\sigma(PX,Z),f\xi\)=-\widetilde g\(\sigma(X,Z),\xi\)
$$
and consequently
\begin{equation}
\label{ddp_sigma:eq}
\begin{array}{l}
\sigma_{i*a}^{\alpha}=\widetilde g\(\sigma(X_{i*},Z_a),\xi_{\alpha}\)=-\widetilde g\(\sigma(X_i,Z_a),\xi_{\alpha*})=\sigma_{ia}^{\alpha*}, \\[2mm]
\sigma_{i*a}^{\alpha*}=-\widetilde g\(\sigma(X_{i*},Z_a),\xi_{\alpha*}\)=\widetilde g\(\sigma(X_i,Z_a),\xi_\alpha)=\sigma_{ia}^\alpha\ .
\end{array}
\end{equation}
Moreover we know that $\widetilde g\(\sigma(PX,Z),FW\)=-X(\ln f)g(Z,W)$. This yields
\begin{equation}
\label{ddp_logf:eq}
\sigma_{ia}^b=PX_i(\ln f)~\delta_{ab} \ {\rm and}\
\sigma_{i*a}^b=X_i(\ln f)~\delta_{ab}.
\end{equation}
By combining \eqref{ddp_eq}, \eqref{ddp_sigma:eq} and \eqref{ddp_logf:eq} we get
\begin{equation}
\begin{array}{c}
\widetilde g\(\sigma(\D,\Dp),\sigma(\D,\Dp)\)= p\sum\limits_{i=1}^h
                \big[ \big(X_i(\ln f)\big)^2-\big(PX_i(\ln f)\big)^2\big]\\[2mm]
                +2\sum\limits_{i=1}^h\sum\limits_{a=1}^p\sum\limits_{\alpha=1}^q
                \big[(\sigma_{ia}^\alpha)^2-(\sigma_{ia}^{\alpha*})^2\big].
\end{array}
\end{equation}
As $\widetilde g\(\sigma(X,Z),f\xi\)=-\widetilde g\(\nabla^\perp_XFZ,\xi\)$ and using the hypothesis
$\nabla_{\D}^\perp\p\Dp\subseteq\p\Dp$ we get $\sigma(\D,\Dp)\subseteq\p\Dp$.
Hence $\sigma_{ia}^\alpha$ and $\sigma_{ia}^{\alpha*}$ vanish. Thus
\begin{equation}
\label{sddp:eq}
\widetilde g\(\sigma(\D,\Dp),\sigma(\D,\Dp)\)= p~g\(\nabla\ln f,\nabla\ln f\).
\end{equation}

Finally, we study $\widetilde g\(\sigma(\Dp,\Dp),\sigma(\Dp,\Dp)\)$.
We write
$$
\sigma(Z_a,Z_b)=\sigma_{ab}^cFZ_c+\sigma_{ab}^\alpha\xi_\alpha+\sigma_{ab}^{\alpha*}\xi_{\alpha*}
$$
and hence
$$
\widetilde g\(\sigma(\Dp,\Dp),\sigma(\Dp,\Dp)\)=-\sum\limits_{a,b,c=1}^p(\sigma_{ab}^c)^2+
        \sum\limits_{a,b=1}^p\sum\limits_{\alpha=1}^q\big[ (\sigma_{ab}^\alpha)^2-(\sigma_{ab}^{\alpha*})^2\big].
$$

As $\widetilde g\(\sigma(Z,W),f\xi\)=-\widetilde g\(\nabla^\perp_ZFW,\xi\)$ and using the hypothesis
$\nabla_{\Dp}^\perp\p\Dp\subseteq\p\Dp$ we get $\sigma(\Dp,\Dp)\subseteq\p\Dp$.
Hence $\sigma_{ab}^\alpha$ and $\sigma_{ab}^{\alpha*}$ vanish. We conclude with
\begin{equation}
\label{dpdp:eq}
\widetilde g\(\sigma(\Dp,\Dp),\sigma(\Dp,\Dp)\)=-\sum\limits_{a,b,c=1}^p(\sigma_{ab}^c)^2\ .
\end{equation}
From these we obtain the theorem.
\endproof

\begin{remark}
If  the manifold $\Na$ in Theorem~{\rm\ref{ineq:th}} is time-like, then \eqref{ineq:eq} shall be replaced by
\begin{equation}
\label{iineq:eq}
S_\sigma \geq 2p {||\nabla\ln f||}_2+{||\sigma_\nu^{\D}||}_2.
\end{equation}
\end{remark}

\begin{remark} {\rm For every $\p R$-warped product $\Ni\times\Na$ in a para-K\"ahler manifold $\widetilde M$,  $\dim \widetilde M\geq \dim \Ni+2\dim \Na$ holds. Thus the smallest codimension is $\dim \Na$.}
\end{remark}

\begin{theorem}
\label{small_codim}
Let $\Ni\times_{f}\Na$ be a $\p R$-warped product in a para-K\"ahler manifold $\widetilde M$.
If $\Na$ is space-like (respectively, time-like) and $\dim \widetilde M= \dim \Ni+2\dim \Na$,
then the second fundamental form of $M$ satisfies
\begin{equation}
\label{ineq_small:eq}
S_\sigma \leq 2p {||\nabla\ln f||}_2\;\; \hbox{ {\rm (respectively,} $S_\sigma \geq 2p {||\nabla\ln f||}_2)$}.
\end{equation}

If the equality sign of \eqref{ineq_small:eq} holds identically,  we have
\begin{equation}
\label{eq} \sigma(\mathcal D,\mathcal D)= \sigma(\mathcal D^\perp,\Dp)=\{0\}.\end{equation}
\end{theorem}
\begin{proof} Inequality \eqref{ineq_small:eq} follows from \eqref{ineq:eq}. When the equality sign holds, \eqref{eq} follows from the proof of Theorem 4.1.
\end{proof}


\section{Exact solutions for a Special PDE's System}

We need the exact solutions of the following PDE system for later use.

\begin{proposition}
\label{PDEsyst}
The non-constant solutions $\psi=\psi(s_1,\ldots,s_h,t_1,\ldots,t_h)$ of the following system of
partial differential equations
\begin{subequations}
\renewcommand{\theequation}{\theparentequation .\alph{equation}}
\label{eq:pde}
\begin{align}
\label{eq:pde1} &
\frac{\partial^2 \psi}{\partial s_i\partial s_j}+\frac{\partial\psi}{\partial s_i}~\frac{\partial\psi}{\partial s_j}
                                        +\frac{\partial\psi}{\partial t_i}~\frac{\partial\psi}{\partial t_j}=0\, ,\\
\label{eq:pde2} &
\frac{\partial^2 \psi}{\partial s_i\partial t_j}+\frac{\partial\psi}{\partial s_i}~\frac{\partial\psi}{\partial t_j}
                                        +\frac{\partial\psi}{\partial t_i}~\frac{\partial\psi}{\partial s_j}=0
                                        \ ,\;\;  i,j=1,\ldots,h\, ,\\
\label{eq:pde3} &
\frac{\partial^2 \psi}{\partial t_i\partial t_j}+\frac{\partial\psi}{\partial t_i}~\frac{\partial\psi}{\partial t_j}
                                        +\frac{\partial\psi}{\partial s_i}~\frac{\partial\psi}{\partial s_j}=0
\end{align}
\end{subequations}
are either given by
\begin{equation}
\label{eq:pde_sol1}
\psi=\frac 12\ln\left|\big[\big(\langle {\mathbf{v}},z \rangle+c_1\big)^2-
            \big(\langle \j {\mathbf{v}},z\rangle+c_2\big)^2\big]\right|,
\end{equation}
where $z=(s_1,s_2,\ldots,s_h,t_1,t_2,\ldots,t_h)$,
${\mathbf{v}}=(a_1,a_2,\ldots,a_h,0,b_2,\ldots,b_h)$ is a constant vector in ${\mathbb{R}}^{2h}$
with $a_1\neq0$, $c_1,c_2\in{\mathbb{R}}$ and $\j{\mathbf{v}}=(0,b_2,\ldots,b_h,a_1,a_2,\ldots,a_h)$;

or given by
\begin{equation}
\label{eq:pde_sol2}
 \psi=\frac 12\ln\left|\big(\langle{\mathbf{v_1}},z\rangle+c\big)\big(\langle{\mathbf{v_2}},z\rangle+d\big)\right|,
\end{equation}
where ${\mathbf{v_1}}=\big(0,a_2,\ldots,a_h,0,\epsilon a_2,\ldots,\epsilon a_h\big)$,
${\mathbf{v_2}}=\big(b_1,\ldots,b_h,-\epsilon b_1,\ldots,-\epsilon b_h\big)$ with $b_1\neq0$, $z$ is as above and $c,d\in{\mathbb{R}}$.

Here $\langle~\, ,~\rangle$ denotes the Euclidean scalar product in ${\mathbb{R}}^{2h}$.
\end{proposition}
\proof
Let us make some notations: $\psi_{s_i}:=\frac{\partial \psi}{\partial s_i} $; $\psi_{s_is_j}:=\frac{\partial^2\psi}{\partial s_i\partial s_j}$,
and similar for $\psi_{t_i}$, $\psi_{s_it_j}$, respectively $\psi_{t_it_j}$. The same notations for any other function.

If in \eqref{eq:pde2} we take $i=j$ one gets $\psi_{s_it_i}=-2\psi_{s_i}\psi_{t_i}$ for all $i=1,\ldots,h$.
Since $\psi$ is non-constant, there exists $i_0$ such that at least one of $\psi_{s_{i_0}}$ or $\psi_{t_{i_0}}$
is different from $0$. Without loss of the generality we suppose $i_0=1$. Both situations yield
$$
e^{2\psi}=\zeta(t_1,s_2,t_2,\ldots,s_h,t_h)+\eta(s_1,s_2,t_2,\ldots,s_h,t_h)\, ,
$$
where $\zeta$ and $\eta$ are functions of $2h-1$ variables such that $\zeta+\eta>0$ on the domain of $\psi$.
It follows that

\begin{equation}
\label{eq:8}
\begin{array}{l}
\psi_{s_1}=\dfrac{\eta_{s_1}}{2(\zeta+\eta)}\ ,\
\psi_{s_1s_1}=\dfrac{\eta_{s_1s_1}(\zeta+\eta)-\eta_{s_1}^2}{2(\zeta+\eta)^2}\ ,\\[2mm]
\psi_{t_1}=\dfrac{\zeta_{t_1}}{2(\zeta+\eta)}\ ,\
\psi_{t_1t_1}=\dfrac{\zeta_{t_1t_1}(\zeta+\eta)-\eta_{t_1}^2}{2(\zeta+\eta)^2} \ .
\end{array}
\end{equation}
Using \eqref{eq:pde1} and \eqref{eq:pde3} we obtain
\begin{equation}
2\eta_{s_1s_1}(\zeta+\eta)=\eta_{s_1}^2-\zeta_{t_1}^2\ ,\quad
2\zeta_{t_1t_1}(\zeta+\eta)=\zeta_{t_1}^2-\eta_{s_1}^2\ .
\end{equation}
Since $\zeta+\eta\neq0$, adding the previous relations, one gets
$$
\eta_{s_1s_1}(s_1,s_2,t_2,\ldots,s_h,t_h)+\zeta_{t_1t_1}(t_1,s_2,t_2,\ldots,s_h,t_h)=0
$$
and hence, there exists a function $F$ depending on $s_2,t_2,\ldots,s_h,t_h$ such that
$$
\begin{array}{l}
\eta_{s_1s_1}(s_1,s_2,t_2,\ldots,s_h,t_h)=~2F(s_2,t_2,\ldots,s_h,t_h)\, , \\[2mm]
\zeta_{t_1t_1}(t_1,s_2,t_2,\ldots,s_h,t_h)=-2F(s_2,t_2,\ldots,s_h,t_h)\ .
\end{array}
$$
At this point one integrates with respect to $s_1$ and $t_1$ respectively and one gets
\begin{equation}
\label{eq:11_12}
\begin{array}{l}
\eta(s_1,s_2,t_2,\ldots,s_h,t_h)=~Fs_1^2+Gs_1+H\, , \\[2mm]
\zeta(t_1,s_2,t_2,\ldots,s_h,t_h)=-Ft_1^2-Kt_1-L\, ,
\end{array}
\end{equation}
where $G,H,L$ and $K$ are functions depending on $s_2,t_2,\ldots,s_h,t_h$ satisfying the
following condition
\begin{equation}
\label{eq:13}
4F(H-L)=G^2-K^2.
\end{equation}
It follows that $\eta+\zeta=\(Fs_1^2+Gs_1+H\)-\(Ft_1^2+Kt_1+L\)$.

{\bf Case 1.}
Suppose $F\neq 0$; being continuous, it preserves constant sign; denote it by $\varepsilon$.
From \eqref{eq:13} we have $H-L=\frac{G^2-K^2}{4F}$ which combined with \eqref{eq:11_12} yields
$$
\eta+\zeta=\varepsilon\left[\Big(\varepsilon\sqrt{\varepsilon F}~s_1+\frac{G}{2\sqrt{\varepsilon F}}\Big)^2-
                    \Big(\varepsilon\sqrt{\varepsilon F}~t_1+\frac{K}{2\sqrt{\varepsilon F}}\Big)^2\right].
$$
We make some notations: $a=\varepsilon\sqrt{\varepsilon F}$, $\gamma=\frac G{2\sqrt{\varepsilon F}}$
and $\delta=\frac K{2\sqrt{\varepsilon F}}$, all of them being functions depending on
$s_2,t_2,\ldots,s_h,t_h$.
We are able to re-write the function $\psi$ as
\begin{equation}
\label{eq:14}
\psi=\frac12\ln\varepsilon\left[(as_1+\gamma)^2-(at_1+\delta)^2\right].
\end{equation}
We compute now
\begin{equation}
\label{eq:15}
\psi_{s_1}=\frac{a(as+\gamma)}{(as_1+\gamma)^2-(at_1+\delta)^2}\ ,\
\psi_{t_1}=\frac{-a(at_1+\delta)}{(as_1+\gamma)^2-(at_1+\delta)^2}
\end{equation}
and for $i\neq1$
\begin{equation}
\label{eq:16_17}
\begin{array}{l}
\displaystyle
\psi_{s_i}=\frac{(as_1+\gamma)(a_{s_i}s_1+\gamma_{s_i})-(at_1+\delta)(a_{s_i}t_1+\delta_{s_i})}{(as_1+\gamma)^2-(at_1+\delta)^2}\, ,  \\[3mm]
\displaystyle
\psi_{t_i}=\frac{(as_1+\gamma)(a_{t_i}s_1+\gamma_{t_i})-(at_1+\delta)(a_{t_i}t_1+\delta_{t_i})}{(as_1+\gamma)^2-(at_1+\delta)^2}\ .
\end{array}
\end{equation}
Computing also $\psi_{s_1s_i}$, we can use \eqref{eq:pde1} for $j=1$, $i=2,\ldots,h$ and obtain
$$
\begin{array}{l}
[a(a_{s_i}s_1+\gamma_{s_i})+a_{s_i}(as_1+\gamma)][(as_1+\gamma)^2-(at_1+\delta)^2]\\
\qquad-a(as_1+\gamma)[(as_1+\gamma)(a_{s_i}s_1+\gamma_{s_i})-(at_1+\delta)(a_{s_i}t_1+\delta_{s_i})]\\
\qquad-a(at_1+\delta)[(as_1+\gamma)(a_{t_i}s_1+\gamma_{t_i})-(at_1+\delta)(a_{t_i}t_1+\delta_{t_i})]=0.
\end{array}
$$
This represents a polynomial in $s_1$ and $t_1$, identically zero, and hence, all its coefficients must vanish.
Analyzing the coefficients for $s_1^3$ and $t_1^3$ we obtain $a_{s_i}=0$ and $a_{t_i}=0$ for all $i=2,\ldots,h$.
Consequently $a$ is a real constant.

Replacing in the previous equation we get
$$
\delta_{s_i}(as_1+\gamma)-\gamma_{s_i}(at_1+\delta)-\gamma_{t_i}(as_1+\gamma)+\delta_{t_i}(at_1+\delta)=0.
$$
Looking at the coefficients of $s_1$ and $t_1$ we have
\begin{equation}
\label{eq:19}
\delta_{s_i}=\gamma_{t_i}\ {\rm and\ }\delta_{t_i}=\gamma_{s_i},\ \forall i=2,\ldots,h.
\end{equation}
Therefore \eqref{eq:16_17} gives
\begin{equation}
\label{eq:20}
\psi_{s_i}=\frac{\gamma_{s_i}(as_1+\gamma)-\delta_{s_i}(at_1+\delta)}{(as_1+\gamma)^2-(at_1+\delta)^2}\ ,\
\psi_{t_i}=\frac{\gamma_{t_i}(as_1+\gamma)-\delta_{t_i}(at_1+\delta)}{(as_1+\gamma)^2-(at_1+\delta)^2}\ .
\end{equation}
We may compute
\begin{equation}
\begin{array}{l}
\displaystyle
\psi_{s_it_j}=\frac{\gamma_{s_it_j}(as_1+\gamma)+\gamma_{s_i}\gamma_{t_j}-\delta_{s_it_j}(at_1+\delta)-\delta_{s_i}\delta_{t_j}}
        {(as_1+\gamma)^2-(at_1+\delta)^2}\qquad\\[3mm]
        \qquad \displaystyle       -2~\frac{[\gamma_{t_j}(as_1+\gamma)-\delta_{t_j}(at_1+\delta)][\gamma_{s_i}(as_1+\gamma)-\delta_{s_i}(at_1+\delta)]}
        {[{(as_1+\gamma)^2-(at_1+\delta)^2}]^2}
\end{array}
\end{equation}
and using \eqref{eq:pde2} with $i,j>1$, we obtain again a polynomial in $s_1$ and $t_1$, identically zero.
By comparing the coefficients of $s_1^3$ and $t_1^3$ we find $\gamma_{s_it_j}=0$ and $\delta_{s_it_j}=0$, for all
$i,j=2,\ldots,h$. It follows that $\gamma_{s_i}$ depend only on $s_2,\ldots,s_h$ and $\delta_{t_i}$ depend only on
$t_2,\ldots,t_h$, for all $i$. From \eqref{eq:19} we know $\gamma_{s_i}=\delta_{t_i}$. Hence, there exist constants
$a_i\in{\mathbb{R}}$ such that $\gamma_{s_i}=\delta_{t_i}=a_i$, $\forall i=2,\ldots,h$. In the same way, there exist
constants $b_i\in{\mathbb{R}}$ such that $\gamma_{t_i}=\delta_{s_i}=b_i$, $\forall i=2,\ldots,h$.
It follows that
\begin{equation}
\label{eq:22}
\begin{array}{l}
\gamma(s_2,t_2\ldots,s_h,t_h)=\sum\limits_{i=2}^ha_is_i+\sum\limits_{i=2}^hb_it_i+c_1, \\[2mm]
\delta(s_2,t_2\ldots,s_h,t_h)=\sum\limits_{i=2}^hb_is_i+\sum\limits_{i=2}^ha_it_i+c_2\ ,\quad c_1,c_2\in{\mathbb{R}}.
\end{array}
\end{equation}
We conclude with
$$
\begin{array}{l}
\psi=\frac12~\ln\varepsilon\big[(as_1+a_2s_2+b_2t_2+\ldots+a_hs_h+b_ht_h+c_1)^2\qquad\\[2mm]
   \qquad \qquad-(at_1+b_2s_2+a_2t_2+\ldots+b_hs_h+a_ht_h+c_2)^2\big].
\end{array}
$$
Hence the solution \eqref{eq:pde_sol1} is obtained with $a_1=a\neq0$.

{\bf Case 2.}
Let us come back to the case $F=0$ (on a certain open set).
From \eqref{eq:13} we immediately find
$\eta+\zeta=Gs_1-Kt_1+H$, where $G,H,K$ are functions depending on $(s_2,\ldots,s_h,t_2,\ldots,t_h)$, and $K=\epsilon G$, $\epsilon=\pm1$.
Thus
$$
\psi=\frac12\ln|(s_1-\epsilon t_1)G+H|.
$$
We have
$$
\psi_{s_1}=\frac G{2[(s_1-\epsilon t_1)G+H]}\ ,\ \psi_{t_1}=-\frac {\epsilon G}{2[(s_1-\epsilon t_1)G+H]},
$$
$$
\psi_{s_i}=\frac{(s_1-\epsilon t_1)G_{s_i}+H_{s_i}}{2[(s_1-\epsilon t_1)G+H]}\ ,\
\psi_{t_i}=\frac{(s_1-\epsilon t_1)G_{t_i}+H_{t_i}}{2[(s_1-\epsilon t_1)G+H]}\ ,\ i=2,\ldots,h,
$$
$$
\psi_{s_is_1}=\frac{G_{s_i}[(s_1-\epsilon t_1)G+H]-G[(s_1-\epsilon t_1)G_{s_i}+H_{s_i})}{2[(s_1-\epsilon t_1)G+H]^2}\ ,\ i=2,\ldots,h.
$$
By applying \eqref{eq:pde1} for $j=1$ and $i=2,\ldots,h$ we obtain
$$
2G_{s_i}[(s_1-\epsilon t_1)G+H]-G[(s_1-\epsilon t_1)G_{s_i}+H_{s_i}]-\epsilon G[(s_1-\epsilon t_1)G_{t_i}+H_{t_i}]=0.
$$
By comparing the coefficients of $s_1$ and $t_1$ we find
\begin{equation}
\label{eq_x}
G(G_{s_i}-\epsilon G_{t_i})=0,\quad 2G_{s_i}H-G(H_{s_i}+\epsilon H_{t_i})=0.
\end{equation}
Since $G\neq0$ we have $G_{t_i}=\epsilon G_{s_i}$. In the sequel, computing
$$
\psi_{s_is_j}=\frac{(s_1-\epsilon t_1)G_{s_is_j}+H_{s_is_j}}{2[(s_1-\epsilon t_1)G+H]}-
  \frac{[(s_1-\epsilon t_1)G_{s_i}+H_{s_i}][(s_1-\epsilon t_1)G_{s_j}+H_{s_j}]}{2[(s_1-\epsilon t_1)G+H]^2}
$$
for $i,j\geq2$, replacing in \eqref{eq:pde1} and comparing the coefficients of $s_1^2$
we find
\linebreak
$G_{s_is_j}=0$. It follows also $G_{s_it_j}=0$ and $G_{t_it_j}=0$. Hence
$$
G(s_2,t_2,\ldots,s_h,t_h)=\sum\limits_{i=2}^ha_i(s_i+\epsilon t_i)+c, \quad a_i,c\in{\mathbb{R}}.
$$
Moreover, $H$ should satisfy
\begin{equation}
\label{eq_y}
2GH_{s_is_j}-G_{s_i}(H_{s_j}-\epsilon H_{t_j})-G_{s_j}(H_{s_i}-\epsilon H_{t_i})=0\, ,
\end{equation}
\begin{equation}
\label{eq_z}
2HH_{s_is_j}-H_{s_i}H_{s_j}+H_{t_i}H_{t_j}=0.
\end{equation}

{\bf Case 2a.}
If $G$ is a non-zero constant $c$ (and this happens when all $a_i$ vanish),
then from the second equation in \eqref{eq_x} we find $H_{s_i}+\epsilon H_{t_i}=0$
for all $i\geq2$. Therefore, $H$ has the form
$$
H(s_2,t_2,\ldots,s_h,t_h)=Q(s_2-\epsilon t_2,\ldots,s_h-\epsilon t_h),
$$
where $Q$ is a function depending only on $h$ variables.
From \eqref{eq_y} we get $H_{s_is_j}=0$ and then $Q$ is an affine function.
Thus
$H=\sum\limits_{i=2}^hb_i(s_i-\epsilon t_i)+d$, with
$b_2,\ldots,b_h,d\in{\mathbb{R}}$. Consequently,
$$
\psi=\frac12\ln\left[\sum\limits_{i=1}^hb_i(s_i-\epsilon t_i)+d\right],\;\; b_1=c\neq0\, .
$$

{\bf Case 2b.}
If there exists at least one $a_i\neq0$, from the second equation in \eqref{eq_x}
we can express $H$ in the form $H=Q G$, where $Q$ is a function on $s_2,t_2,\ldots,s_h,t_h$.
Then, for every $i\geq2$,
$$
H_{s_i}+\epsilon H_{t_i}=2a_iQ+G(Q_{s_i}+Q_{t_i})\, ,
$$
which combined with \eqref{eq_x} gives
$Q_{s_i}+\epsilon Q_{t_i}=0$.
Thus, $Q=Q(s_2-\epsilon t_2,\ldots,s_h-\epsilon t_h)$.
Using \eqref{eq_y}, it follows that $Q$ is an affine function and hence
$H=\sum\limits_{i=2}^hb_i(s_i-\epsilon t_i)+d$, with
$b_2,\ldots,b_h,d\in{\mathbb{R}}$. Consequently,
$$\psi=\frac12\ln\Big\{\Big[\sum\limits_{i=1}^hb_i(s_i-\epsilon t_i)+d\Big]\Big[\sum\limits_{j=2}^ha_i(s_i+\epsilon t_i)+c\Big]\Big\}$$
with $b_1=1$. This completes the proof.
\endproof


\section{$\p R$-warped products in ${\p}^{h+p}$ satisfying $S_\sigma=2p{||\nabla\ln f||}_2$}

In the following, we will use letters $i,j,k$ for indices running from $1$ to $h$;  $a,b,c$ for indices from $1$ to $p$; and $A,B$ for indices between $1$ and $m$ with $m=h+p$.

On ${\mathbb{E}}^{2(h+p)}_{h+p}$ we consider the global coordinates $(x_i,x_{h+a},y_i,y_{h+a})$ and the canonical flat para-K\"ahler structure defined as above.

\begin{proposition}
\label{prop_v}
Let $M=\Ni\times_f\Na$ be a space-like $\p R$-warped product in the  para-K\"ahler $(h+p)$-plane $\p^{h+p}$ with $h=\frac{1}{2}\dim \Ni$ and $p=\dim \Na$. If $M$ satisfies the equality case of \eqref{ineq_small:eq} identically,  then
\begin{itemize}
\item $\Ni$ is a totally geodesic submanifold in $\p^{h+p}$, and hence it is congruent to an open part of $\p^h$;
\item $\Na$ is a totally umbilical submanifold in $\p^{h+p}$.
\end{itemize}
Moreover, if $\Na$ is a real space form of constant curvature $k$, then the warping function $f$ satisfies
${||\nabla f||}_2=k$.
\end{proposition}
\proof 
Under the hypothesis, we know from the proof of Theorem \ref{ineq:th} that the second fundamental form satisfies
$$\sigma(\mathcal D,\mathcal D)=\sigma(\mathcal D^\perp,\mathcal D^\perp)=\{0\}.$$ 
On the other hand, since $M=\Ni\times_f\Na$ is a warped product, $\Ni$ is totally geodesic and $\Na$ is totally umbilical in $M$.
Thus we have the first two statements.

The last statement of the proposition can be proved as follows.
If $R^\perp$ denotes the Riemann curvature tensor of $\Na$, then we have
$$
R_{ZV}W=R^\perp_{ZV}W-{||\nabla \ln f||}_2\big(g(V,W)Z-g(Z,W)V\big)
$$
for any $Z,V,W$ tangent to $\Na$. See for details \cite[page 210]{cm:ONeil} (pay attention to the sign; see also page 74).
If $\Na$ is a space form of constant curvature $k$, then $R$ takes the form
\begin{equation}
\label{eq:R_k}
R_{ZV}W=\left(\frac k{f^2}-{||\nabla\ln f||}_2\right)\big(g(V,W)Z-g(Z,W)V\big).
\end{equation}
The equation of Gauss may be written, for vectors tangent to $\Na$, as
$$
g\(R_{ZV}W,U\)=\<\widetilde R_{ZV}W,U\>+\<\sigma(V,W),\sigma(Z,U)\>-\<\sigma(Z,W),\sigma(V,U)\> \ .
$$
Since the ambient space is flat and \ $\sigma(\Dp,\Dp)=0$ due to the equality in \eqref{ineq_small:eq}, it follows that
$g(R_{ZV}W,U)=0$. Combining this with \eqref{eq:R_k} gives ${||\nabla\ln f||}_2=\frac k{f^2}$.  This gives
the statement.
\endproof

{\it Para-complex numbers} were introduced by  Graves in 1845 \cite{graves} as a generalization of complex numbers. Such numbers have the expression $v=x+\j y$, where $x,y$ are real numbers and $\j$ satisfies $\j^{2}=1,\,\j\ne 1$. 
The conjugate of $v$ is $\bar v=x-\j y$. The multiplication of two para-complex numbers is defined by
$$(a+\j b)(s+\j t)=(as+bt)+\j(at+bs).$$

For each natural number $m$, we put $\mathbb D^{m}=\{(x_1+\j y_1,\ldots,x_m+\j y_m) : x_i, y_i\in{\mathbb{R}}\}$. With respect to the  multiplication of  para-complex numbers and the canonical flat metric,  $\mathbb D^{m}$ is a flat para-K\"ahler manifold of dimension $2m$. 
Once we  identify $(x_1+\j y_1,\ldots,x_m+\j y_m)\in \mathbb D^{m}$ with $(x_1,\ldots,x_m,y_1,\ldots,y_m)\in {\mathbb{E}}^{2m}_m$, we may identify $\mathbb D^{m}$ with the para-K\"ahler $m$-plane $\p^{m}$ in a natural way.

In the following we denote by $\mathbb S^{p}, \mathbb E^{p}$ and  $\mathbb  H^{p}$ the unit $p$-sphere, the Euclidean $p$-space and the unit hyperbolic $p$-space, respectively.

\begin{theorem}
Let $\Ni\times_f\Na$ be a space-like $\p R$-warped product  in the  para-K\"ahler $(h+p)$-plane $\p^{h+p}$ with $h=\frac{1}{2}\dim \Ni$ and $p=\dim \Na$. Then we have \begin{align}\label{IN}S_\sigma\leq 2p{||\nabla\ln f||}_2.\end{align} 
The equality sign of \eqref{IN} holds identically if and only if  $\Ni$ is an open part of a para-K\"ahler $h$-plane, $\Na$ is an open part of $\mathbb S^{p},\, \mathbb E^{p}$ or $\mathbb H^{p}$,  and the immersion  is given by one of the following:

{\bf 1.} $\Phi:D_{1}\times_f {\mathbb{S}}^p\longrightarrow {\p}^{h+p}$;
\begin{equation}
\label{case1}
\begin{array}{c}
\Phi(z,w)=\text{\small$\Bigg($}z_1+\bar v_1(w_0-1)\sum\limits_{j=1}^hv_jz_j,\ldots,z_h+\bar v_h(w_0-1)\sum\limits_{j=1}^hv_jz_j,\\
\qquad w_1\sum\limits_{j=1}^h{\rm j} v_jz_j,\ldots,w_p\sum\limits_{j=1}^h\j v_jz_j \text{\small$\Bigg)$},\;\; h\geq 2,
\end{array}
\end{equation}
with warping function $$f=\sqrt{\langle \bar v,z\rangle^2-\langle \j \bar v,z\rangle^2},$$
where $v=(v_1,\ldots,v_h)\in{\mathbb{S}}^{2h-1}\subset {\mathbb{D}}^h$, $ w=(w_0,w_1,\ldots,w_p)\in{\mathbb{S}}^p$, $ z=(z_1,\ldots,z_h)\in D_{1}$ 
and $D_{1}=\left\{z\in{\mathbb{D}}^h : \langle \bar v,z\rangle^2>\langle \j \bar v,z\rangle^2\right\}$.

{\bf 2.} $\Phi:D_{1}\times_f {\mathbb{H}}^p\longrightarrow {\p}^{h+p}$;
\begin{equation}
\label{case2}
\begin{array}{c}
\Phi(z,w)=\text{\small$\Bigg($}z_1+\bar v_1(w_0-1)\sum\limits_{j=1}^hv_jz_j,\ldots,z_h+\bar v_h(w_0-1)\sum\limits_{j=1}^hv_jz_j,\\
\qquad w_1\sum\limits_{j=1}^h\j v_jz_j,\ldots,w_p\sum\limits_{j=1}^h\j v_jz_j \text{\small$\Bigg)$},\;\; h\geq 1,
\end{array}
\end{equation}
with the warping function $f=\sqrt{\langle \bar v,z\rangle^2-\langle \j \bar v,z\rangle^2}$, where $v=(v_1,\ldots,v_h)\in{\mathbb{H}}^{2h-1}\subset {\mathbb{D}}^h$,
$w=(w_0,w_1,\ldots,w_p)\in{\mathbb{H}}^p$  and 
$z=(z_1,\ldots,z_h)\in D_{1}$.

{\bf 3.} $\Phi(z,u):D_{1}\times_f {\mathbb{E}}^p\longrightarrow {\p}^{h+p}$;
\begin{equation}
\label{case3}
\begin{array}{c}
\Phi(z,u)=\text{\small$ \Bigg($}z_1+\dfrac{\bar v_1}{2}\Big(\sum\limits_{a=1}^pu_a^2\Big)\sum\limits_{j=1}^hv_jz_j,\ldots,
        z_h+\dfrac{\bar v_h}{2}\Big(\sum\limits_{a=1}^pu_a^2\Big)\sum\limits_{j=1}^hv_jz_j,\\
        u_1\sum\limits_{j=1}^h\j v_jz_j,\ldots,u_p\sum\limits_{j=1}^h\j v_jz_j \text{\small$ \Bigg)$},\;\; h\geq 2,
\end{array}
\end{equation}
with the warping function $f=\sqrt{\langle \bar v,z\rangle^2-\langle \j \bar v,z\rangle^2}$, where $v=(v_1,\ldots,v_h)$ is a light-like vector in ${\mathbb{D}}^h$,
$z=(z_1,\ldots,z_h)\in D_{1}$ and $u=(u_1,\ldots,u_p)\in{\mathbb{E}}^p,$

Moreover, in this case, each leaf $\, {\mathbb{E}}^p$ is quasi-minimal in ${\p}^{h+p}$.

{\bf 4.} $\Phi(z,u):D_{2}\times_f {\mathbb{E}}^p\longrightarrow {\p}^{h+p}$;
\begin{equation}
\label{case4}
\begin{array}{c}  \Phi(z,u)=\text{\small$\Bigg( $}z_1+\dfrac{v_1}{2}\! \sum\limits_{a=1}^pu_a^2,\ldots,
        z_h+\dfrac{v_h}{2}\!\sum\limits_{a=1}^pu_a^2,
        \dfrac{v_0}{2}u_1,\ldots,\dfrac{v_0}{2}u_p \text{\small$\Bigg)$},\; h\geq 1, \end{array}
\end{equation}
with the warping function $f=\sqrt{-\langle v,z\rangle}$, where $v_0=\sqrt{b_1}+\epsilon\j\sqrt{b_1}$ with $b_1>0$,  $D_{2}=\{z\in{\mathbb{D}}^h:\langle v,z\rangle<0\}$,
$v=(v_1,\ldots,v_h)=(b_1+\epsilon\j b_1,\ldots,b_h+\epsilon\j b_h)$, $\epsilon=\pm1$,
$z=(z_1,\ldots,z_h)\in D_2$ and $ u=(u_1,\ldots,u_p)\in{\mathbb{E}}^p$.

In each of the four cases the warped product is minimal in ${\mathbb{E}}^{2(h+p)}_{h+p}$.
\end{theorem}
\begin{proof} Inequality \eqref{IN} is already given in Theorem 4.4. From now on, let us assume that 
 $\Phi:\Ni\times_f\Na\longrightarrow{\p}^{m}$ is a space-like $\p R$-warped product satisfying the
equality in \eqref{IN} with $m=h+p$. Then it follows that $\nu=0$ and hence
\begin{equation}
\label{eq:tgtu}
\sigma(X,Y)=0,\ \sigma(Z,W)=0,\  \sigma(X,Z)=\big(PX(\ln f)\big)FZ,
\end{equation}
for all $X,Y$ tangent to $\Ni$ and $Z,W$ tangent to $\Na$.
Thus, $\Ni$ is totally geodesic in ${\p}^m$ and $\Na$ is totally umbilical ${\p}^{m}$.

As $\Ni$ is invariant and totally geodesic in ${\p}^m$, it is congruent
with ${\p}^h$ with the canonical (induced) para-K\"ahler structure \cite{cm:Chen11}.
On ${\mathbb{E}}^{2h}_h$ we may choose global coordinates $s=(s_1,\ldots,s_h)$ and $t=(t_1,\ldots,t_h)$
such that
\begin{equation}
\label{A1}
g_\top=-\sum\limits_{i=1}^hds_i^2+\sum\limits_{i=1}^hdt_i^2,\;\; \p\dsi=\dti,\;\; \p\dti=\dsi,
\end{equation}
 for $i=1,\ldots,h$. 
 
 Let us put $\dsi=\frac\partial{\partial s_i}\ ,$
$\dti=\frac\partial{\partial t_i}$ and so on.

\smallskip

Now, we study the case $p>1$.

\smallskip

Since $\Na$ is a space-like totally umbilical, non-totally geodesic submanifold in ${\p}^m$,
it is congruent (cf. \cite{cm:AKK96}, \cite[Proposition 3.6]{cm:Chen11})
\begin{itemize}
\item either to the Euclidean $p$-sphere ${\mathbb{S}}^p$,
\item or to the hyperbolic $p$-plane ${\mathbb{H}}^p$,
\item or to a flat quasi-minimal submanifold ${\mathbb{E}}^p$.
\end{itemize}

In what follows we discuss successively, all the three situations.

On ${\mathbb{S}}^p$ we consider spherical coordinates $u=(u_1,\ldots,u_p)$ such that
the metric $g_\perp$ is expressed by
\begin{equation}
\label{A3}
g_\perp=du_1^2+\cos^2u_1du_2^2+\ldots+\cos^2u_1\ldots\cos^2u_{p-1}du_p^2.
\end{equation}
Thus, the warped metric on $M$ is given by
$$
g=g_\top(s,t)+f^2(s,t)g_\perp(u).
$$
Then the Levi Civita connection $\nabla$ of $g$ satisfies
\begin{subequations}
\renewcommand{\theequation}{\theparentequation .\alph{equation}}
\label{eq:LCSp}
\begin{align}
\label{A5} & \nabla_{\dsi}\dsj=0\ ,\ \nabla_{\dsi}\dtj=0\ ,\ \nabla_{\dti}\dtj=0,\\
\label{A6} &\nabla_{\dsi}\dua=\frac{f_{s_i}}f~\dua\, ,\ \nabla_{\dti}\dua=\frac{f_{t_i}}f~\dua,\\
\label{A7} &\nabla_{\dua}\dub=-\tan u_a\dub\quad(a<b),\\
\label{A8} & \nabla_{\dua}\dua=\prod\limits_{b=1}^{a-1}\cos^2u_b\sum\limits_{i=1}^h\big(ff_{s_i}\dsi-ff_{t_i}\dti\big)\\
\nonumber & \qquad\qquad + \sum\limits_{b=1}^{a-1}\big(\sin u_b\cos u_b\cos^2u_{b+1}\ldots\cos^2u_{a-1}\big)\dub,
\end{align}
\end{subequations}
for $ i,j =1,\ldots,h$ and $a, b=1,\ldots,p$.

From now on we put $\psi=\ln f$.
Using the relations above, we find that the Riemann curvature  tensor $R$ satisfies
\begin{equation}
\label{eq:curb}
\begin{array}{l}
R(\dsi,\dua)~\dsj=\left(\dfrac{\partial^2\psi}{\partial{s_i}\partial{s_j}}+\dfrac{\partial\psi}{\partial{s_i}}\dfrac{\partial\psi}{\partial{s_j}}\right)\dua\\[2mm]
R(\dsi,\dua)~\dtj=\left(\dfrac{\partial^2\psi}{\partial{s_i}\partial{t_j}}+\dfrac{\partial\psi}{\partial{s_i}}\dfrac{\partial\psi}{\partial{t_j}}\right)\dua\\[2mm]
R(\dti,\dua)~\dtj=\left(\dfrac{\partial^2\psi}{\partial{t_i}\partial{t_j}}+\dfrac{\partial\psi}{\partial{t_i}}\dfrac{\partial\psi}{\partial{t_j}}\right)\dua\ .
\end{array}
\end{equation}
Moreover we have
$$
\sigma(\dsi,\dua)=\frac{\partial\psi}{\partial{t_i}}~F\dua,\ \sigma(\dti,\dua)=\frac{\partial\psi}{\partial{s_i}}~F\dua.
$$
Applying Gauss' equation we find
$$
\widetilde g\(\widetilde R_{XZ}Y,W\)=g\(R_{XZ}Y,W\)+\widetilde g\(\sigma(X,Y),\sigma(Z,W)\)-\widetilde g\(\sigma(X,W),\sigma(Y,Z)\),
$$  for $X,Y$ tangent to $\Ni$ and $Z,W$ tangent to $\Na$. Using \eqref{eq:tgtu} and \eqref{eq:curb} we get
\begin{equation}
\label{eq:warping}
\begin{array}{l}
\displaystyle
\frac{\partial^2\psi}{\partial{s_i}\partial{s_j}}+\frac{\partial\psi}{\partial{s_i}}\frac{\partial\psi}{\partial{s_j}}
            +\frac{\partial\psi}{\partial{t_i}}\frac{\partial\psi}{\partial{t_j}}=0\\[3mm]
\displaystyle
\frac{\partial^2\psi}{\partial{s_i}\partial{t_j}}+\frac{\partial\psi}{\partial{s_i}}\frac{\partial\psi}{\partial{t_j}}
            +\frac{\partial\psi}{\partial{t_i}}\frac{\partial\psi}{\partial{s_j}}=0\\[3mm]
\displaystyle
\frac{\partial^2\psi}{\partial{t_i}\partial{t_j}}+\frac{\partial\psi}{\partial{s_i}}\frac{\partial\psi}{\partial{s_j}}
            +\frac{\partial\psi}{\partial{t_i}}\frac{\partial\psi}{\partial{t_j}}=0,\ i=1,\ldots,h.
\end{array}
\end{equation}
Let us first consider the case $h\geq2$.

By applying Proposition \ref{PDEsyst} (case 1, in the proof), we know that there exists a constant vector
$v=(a_1,a_2,\ldots,a_h,0,b_2,\ldots,b_h)$, with $a_1>0$, such that
$$
\psi=\frac12\ln\left[\langle \bar v,z\rangle^2-\langle \j \bar v,z\rangle^2\right],
$$
where $z=(s_1,\ldots,s_h,t_1,\ldots,t_h)$ and $\langle~,~\rangle$ denotes the pseudo-Euclidean product in ${\mathbb{E}}^{2h}_h$.
If $a_1<0$ we are allowed to make the isometric transformation in ${\mathbb{E}}^{2h}_h$: $s_1\mapsto-s_1$ and $t_1\mapsto-t_1$.
In the sequel, we apply Gauss' formula
$$
\widetilde\nabla_{\Phi_*U}\Phi_*V=\Phi_*\nabla_UV+\sigma(U,V), \ \forall U,V\in\chi(M),
$$
where $\Phi_*$ denotes the differential of the map $\Phi$.
Taking $U,V\in\D$ we obtain
\begin{equation}
\label{embed_DD}
\begin{array}{l}
\displaystyle
\frac{\partial^2x_A}{\partial s_i\partial s_j}=\frac{\partial^2x_A}{\partial s_i\partial t_j}=\frac{\partial^2x_A}{\partial t_i\partial t_j}=0\\[3mm]
\displaystyle
\frac{\partial^2y_A}{\partial s_i\partial s_j}=\frac{\partial^2y_A}{\partial s_i\partial t_j}=\frac{\partial^2y_A}{\partial t_i\partial t_j}=0.
\end{array}
\end{equation}
For $U\in\D$ and $V\in\Dp$ we have
\begin{equation}
\label{embed_DDp}
\begin{array}{l}
\displaystyle
\frac{\partial^2x_A}{\partial s_i\partial u_a}=\psi_{s_i}\frac{\partial x_A}{\partial u_a}+\psi_{t_i}\frac{\partial y_A}{\partial u_a}\ ,\quad
\displaystyle
\frac{\partial^2x_A}{\partial t_i\partial u_a}=\psi_{t_i}\frac{\partial x_A}{\partial u_a}+\psi_{s_i}\frac{\partial y_A}{\partial u_a}\\[3mm]
\displaystyle
\frac{\partial^2y_A}{\partial s_i\partial u_a}=\psi_{s_i}\frac{\partial y_A}{\partial u_a}+\psi_{t_i}\frac{\partial x_A}{\partial u_a}\ ,\quad
\displaystyle
\frac{\partial^2y_A}{\partial t_i\partial u_a}=\psi_{t_i}\frac{\partial y_A}{\partial u_a}+\psi_{s_i}\frac{\partial x_A}{\partial u_a}.
\end{array}
\end{equation}
Finally, taking $U,V\in\Dp$ we obtain
\begin{equation}
\label{embed_DpDp}
\begin{array}{l}
\displaystyle
\frac{\partial^2x_A}{\partial u_a\partial u_b}=-\tan u_a \frac{\partial x_A}{\partial u_b}\ ,\quad
\frac{\partial^2y_A}{\partial u_a\partial u_b}=-\tan u_a \frac{\partial y_A}{\partial u_b}\ ,\ a<b\, ,\\[2mm]
\displaystyle
\frac{\partial^2x_A}{\partial u_a^2}=\prod\limits_{b=1}^{a-1}\cos^2u_b\sum\limits_{j=1}^h
        \left(ff_{s_j}\frac{\partial x_A}{\partial s_j}-ff_{t_j}\frac{\partial x_A}{\partial t_j}\right)\quad \\[2mm]
 \qquad \qquad     \displaystyle
         +\sum\limits_{b=1}^{a-1}\(\sin u_b\cos u_b\cos^2u_{b+1}\ldots\cos^2u_{a-1}\)\frac{\partial x_A}{\partial u_b}\, ,\\[2mm]
\displaystyle
\frac{\partial^2y_A}{\partial u_a^2}=\prod\limits_{b=1}^{a-1}\cos^2u_b\sum\limits_{j=1}^h
        \left(ff_{s_j}\frac{\partial y_A}{\partial s_j}-ff_{t_j}\frac{\partial y_A}{\partial t_j}\right)\quad \\[2mm]
 \qquad \qquad     \displaystyle
         +\sum\limits_{b=1}^{a-1}\(\sin u_b\cos u_b\cos^2u_{b+1}\ldots\cos^2u_{a-1}\)\frac{\partial y_A}{\partial u_b}\, .
\end{array}
\end{equation}

From \eqref{embed_DD} we get
\begin{equation}
\label{sol:DD}
\begin{array}{l}
x_A(s,t,u)=\sum\limits_1^h\lambda^j_A(u)s_j+\sum\limits_1^h\rho_A^j(u)t_j+C_A(u)\, ,\\[2mm]
y_A(s,t,u)=\sum\limits_1^h\tilde\rho_A^j(u)s_j+\sum\limits_1^h\tilde\lambda_A^j(u)t_j+\tilde C_A(u)\, .
\end{array}
\end{equation}

By combining \eqref{embed_DDp} with \eqref{sol:DD} we obtain
\begin{equation}
\label{eq:17}
\begin{array}{rl}
\dfrac{\partial\tilde\lambda_A^i}{\partial u_a}=\dfrac{\partial\lambda_A^i}{\partial u_a}&=
    \psi_{s_i}\left[\dfrac{\partial \lambda_A^j}{\partial u_a}(u)s_j+\dfrac{\partial \rho_A^j}{\partial u_a}(u)t_j+\dfrac{\partial C_A}{\partial u_a}\right]\\
  &\hskip.3in  +\psi_{t_i}\left[\dfrac{\partial \rho_A^j}{\partial u_a}(u)s_j+\dfrac{\partial \lambda_A^j}{\partial u_a}(u)t_j+\dfrac{\partial \tilde C_A}{\partial u_a}\right]\, ,\\[2mm]
\dfrac{\partial\tilde\rho_A^i}{\partial u_a}=\dfrac{\partial\rho_A^i}{\partial u_a}&=
    \psi_{t_i}\left[\dfrac{\partial \lambda_A^j}{\partial u_a}(u)s_j+\dfrac{\partial \rho_A^j}{\partial u_a}(u)t_j+\dfrac{\partial C_A}{\partial u_a}\right]\\
  & \hskip.3in +\psi_{s_i}\left[\dfrac{\partial \rho_A^j}{\partial u_a}(u)s_j+\dfrac{\partial \lambda_A^j}{\partial u_a}(u)t_j+\dfrac{\partial \tilde C_A}{\partial u_a}\right].
\end{array}
\end{equation}

For $i=1$ we have
$$
\begin{array}{l}
\psi_{s_1}=\dfrac{a_1\big(a_1s_1+\sum\limits_2^ha_js_j+\sum\limits_2^hb_jt_j\big)}
    {\big(a_1s_1+\sum\limits_2^ha_js_j+\sum\limits_2^hb_jt_j\big)^2-\big(a_1t_1+\sum\limits_2^ha_jt_j+\sum\limits_2^hb_js_j\big)^2}\, ,\\[2mm]
\psi_{t_1}=\dfrac{-a_1\big(a_1t_1+\sum\limits_2^ha_jt_j+\sum\limits_2^hb_js_j\big)}
    {\big(a_1s_1+\sum\limits_2^ha_js_j+\sum\limits_2^hb_jt_j\big)^2-\big(a_1t_1+\sum\limits_2^ha_jt_j+\sum\limits_2^hb_js_j\big)^2}\ .
\end{array}
$$
Substituting in \eqref{eq:17} we find polynomials in $s$ and $t$. Comparing the coefficients corresponding to
$s_1s_i$ and $s_1t_i$, $i>1$, we find
\begin{equation}
\label{eq:20-21}
\begin{array}{l}
\lambda_A^i(u)=\frac{a_i}{a_1}\lambda_A(u)+\frac{b_i}{a_1}\rho_A(u)+\frac{c_A^i}{a_1}\ ,\
\rho_A^i(u)=\frac{b_i}{a_1}\lambda_A(u)+\frac{a_i}{a_1}\rho_A(u)+\frac{d_A^i}{a_1}\
\end{array}
\end{equation}
for $\ i=2,\ldots,h$, and $\lambda_A^1(u)=\lambda_A(u)$, $\rho_A^1(u)=\rho_A(u)$, where $c_A^i,d_A^i\in{\mathbb{R}}$.

Comparing the coefficients of $s_1$ and $t_1$ we find that $C_A$ and $\tilde C_A$ are constants, and applying a
suitable translation in ${\mathbb{E}}^{2m}_m$ if necessary, one may suppose $C_A=0$ and $\tilde C_A=0$,  $A=1,\ldots,m$.
Replacing in \eqref{sol:DD} and taking into account \eqref{eq:17} we get
\begin{equation}
\label{sol:DDp}
\begin{array}{l}
x_A(s,t,u)=\frac1{a_1}\lambda_A(u)\(a_1s_1+\sum\limits_2^ha_js_j+\sum\limits_2^hb_jt_j\)\\[2mm]
    \qquad+    \frac1{a_1}\rho_A(u)\(a_1t_1+\sum\limits_2^ha_jt_j+\sum\limits_2^hb_js_j\)
     +\frac1{a_1}\(\sum\limits_2^hc_A^js_j+\sum\limits_2^hd_A^jt_j\),
     \qquad
\end{array}
\end{equation}
$$
\begin{array}{l}
y_A(s,t,u)=\frac1{a_1}\lambda_A(u)\(a_1t_1+\sum\limits_2^ha_jt_j+\sum\limits_2^hb_js_j\)\\[2mm]
    +    \frac1{a_1}\rho_A(u)\(a_1s_1+\sum\limits_2^ha_js_j+\sum\limits_2^hb_jt_j\)
     +\frac1{a_1}\(\tilde d_As_1+\tilde c_At_1+\sum\limits_2^h\tilde d_A^js_j+\sum\limits_2^h\tilde c_A^jt_j\),
\end{array}
$$
where $\tilde c_A$, $\tilde d_A$, $\tilde c_A^i$ and $\tilde d_A^i$ are real numbers.
The third equation in \eqref{embed_DpDp} for $a=1$ gives
$$
\begin{array}{l}
\dfrac{\partial^2x_A}{\partial u_1^2}=\(a_1s_1+\sum\limits_2^h a_js_j+\sum\limits_2^h b_jt_j\)\left[
        a_1\frac{\partial x_A}{\partial s_1}+\sum\limits_2^ha_j\frac{\partial x_A}{\partial s_j}-
        \sum\limits_2^hb_j\frac{\partial x_A}{\partial t_j}\right]\\[2mm]
   \qquad + \(a_1t_1+\sum\limits_2^h a_jt_j+\sum\limits_2^h b_js_j\)\left[
        a_1\frac{\partial x_A}{\partial t_1}+\sum\limits_2^ha_j\frac{\partial x_A}{\partial t_j}-
        \sum\limits_2^hb_j\frac{\partial x_A}{\partial s_j}\right]
\end{array}
$$
which combined with the first equation in \eqref{sol:DDp} yields
\begin{equation}
\label{S:u1u1}
\begin{array}{l}
\(a_1s_1+\sum\limits_2^h a_js_j+\sum\limits_2^h b_jt_j\)\left[
   \frac{\partial^2\lambda_A}{\partial u_1^2}(u)+\langle v,v\rangle \lambda_A(u)+D_A\right]\\[2mm]
+\(a_1t_1+\sum\limits_2^h a_jt_j+\sum\limits_2^h b_js_j\)\left[
   \frac{\partial^2\rho_A}{\partial u_1^2}(u)+\langle v,v\rangle \rho_A(u)+\tilde D_A\right]=0,
\end{array}
\end{equation}
where $D_A=\sum\limits_2^h(a_jc_A^j-b_jd_A^j)$ and $\tilde D_A=\sum\limits_2^h(a_jd_A^j-b_jc_A^j)$.

Since ${||\nabla f||}_2=-a_1^2-\sum\limits_2^ha_j^2+\sum\limits_2^hb_j^2$,
Proposition~\ref{prop_v} implies $\langle v,v\rangle=1$. Hence, considering in \eqref{S:u1u1}
the coefficients of $s_1$ and $t_1$ one obtains the following PDEs:
\begin{equation}
\frac{\partial^2\lambda_A}{\partial u_1^2}(u)+\lambda_A(u)-D_A=0,\quad
\frac{\partial^2\rho_A}{\partial u_1^2}(u)+\rho_A(u)-\tilde D_A=0.
\end{equation}
We immediately get
\begin{equation}
\label{lambda_rho}
\begin{array}{l}
\lambda_A(u)=\cos u_1\Theta_A^{(1)}(u_2,\ldots,u_p)+\sin u_1D_A^{(1)}(u_2,\ldots,u_p)+D_A, \\[2mm]
\rho_A(u)=\cos u_1\tilde\Theta_A^{(1)}(u_2,\ldots,u_p)+\sin u_1\tilde D_A^{(1)}(u_2,\ldots,u_p)+\tilde D_A
\end{array}
\end{equation}
where $\Theta_A^{(1)}$, $D_A^{(1)}$, $\tilde\Theta_A^{(1)}$ and $\tilde D_A^{(1)}$ are functions depending on $u_2,\ldots,u_p$.
The first equation in \eqref{embed_DpDp} for $a=1$ gives
$
\frac{\partial^2x_A}{\partial u_1\partial u_b}=-\tan u_1\frac{\partial x_A}{\partial u_b}\ ,\ b>1
$
which combined with \eqref{sol:DDp} yields
$$
\dfrac{\partial^2\lambda_A}{\partial u_1\partial u_b}=-\tan u_1\frac{\partial \lambda_A}{\partial u_b}\ ,\
\dfrac{\partial^2\rho_A}{\partial u_1\partial u_b}=-\tan u_1\frac{\partial \rho_A}{\partial u_b}\ .
$$
Using \eqref{lambda_rho}, we get
$\frac{\partial D_A^{(1)}}{\partial u_b}=0$, and $\frac{\partial\tilde D_A^{(1)}}{\partial u_b}=0$, $\forall b>1$, hence $D_A^{(1)}$ and $\tilde D_A^{(1)}$ are real constants.

Returning to the third equation in \eqref{embed_DpDp} with $a=2$ we get
$$
\begin{array}{l}
\dfrac{\partial^2x_A}{\partial u_2^2}=\cos^2u_1\(a_1s_1+\sum\limits_2^h a_js_j+\sum\limits_2^h b_jt_j\)\left[
        a_1\dfrac{\partial x_A}{\partial s_1}+\sum\limits_2^ha_j\dfrac{\partial x_A}{\partial s_j}-
        \sum\limits_2^hb_j\dfrac{\partial x_A}{\partial t_j}\right]\\[2mm]
   \qquad +\cos^2u_2 \(a_1t_1+\sum\limits_2^h a_jt_j+\sum\limits_2^h b_js_j\)\left[
        a_1\frac{\partial x_A}{\partial t_1}+\sum\limits_2^ha_j\dfrac{\partial x_A}{\partial t_j}-
        \sum\limits_2^hb_j\dfrac{\partial x_A}{\partial s_j}\right]\\[2mm]
     \qquad   +\sin u_1\cos u_1\dfrac{\partial x_A}{\partial u_1}\ .
\end{array}
$$
This relation together with \eqref{sol:DDp} yield a polynomial in $s$ and $t$, and considering the coefficients of $s_1$
and $t_1$ respectively, we obtain
$$
\begin{array}{l}
\dfrac{\partial^2\lambda_A}{\partial u_2^2}-\sin u_1\cos u_1\frac{\partial \lambda_A}{\partial u_1}+(\cos^2 u_1)\lambda_A-D_A\cos^2 u_1 =0\, , \\[2mm]
\dfrac{\partial^2\rho_A}{\partial u_2^2}-\sin u_1\cos u_1\frac{\partial \rho_A}{\partial u_1}+(\cos^2 u_1)\rho_A-\tilde D_A\cos^2 u_1 =0\, .
\end{array}
$$
Using \eqref{lambda_rho} one gets
$$
\frac{\partial ^2\Theta_A^{(1)}}{\partial u_2^2}+\Theta_A^{(1)}=0\ , \
\frac{\partial ^2\tilde \Theta_A^{(1)}}{\partial u_2^2}+\tilde\Theta_A^{(1)}=0
$$
with the solutions
$$
\begin{array}{l}
\Theta_A^{(1)}=\cos u_2\Theta_A^{(2)}(u_3,\ldots,u_p)+\sin u_2D_A^{(2)}(u_3,\ldots,u_p)\, ,\\[2mm]
\tilde\Theta_A^{(1)}=\cos u_2\tilde\Theta_A^{(2)}(u_3,\ldots,u_p)+\sin u_2\tilde D_A^{(2)}(u_3,\ldots,u_p),
\end{array}
$$
where $\Theta_A^{(2)}$, $D_A^{(2)}$, $\tilde\Theta_A^{(2)}$ and $\tilde D_A^{(2)}$ are functions depending on $u_3,\ldots,u_p$.
Continuing such procedure sufficiently many times, we find
\begin{equation}
\label{lambda_rho_1}
\begin{array}{l}
\lambda_A(u)=D_A^{(0)}\cos u_1\ldots\cos_{p-1}\cos u_p+
                  D_A^{(p)}\cos u_1\ldots\cos_{p-1}\sin u_p\quad\\[2mm]
              \quad  + D_A^{(p-1)}\cos u_1\ldots\sin_{p-1}+\ldots+
                  D_A^{(2)}\cos u_1\sin u_1+D_A^{(1)}\sin u_1+D_A\, ,\\[3mm]
\rho_A(u)=\tilde D_A^{(0)}\cos u_1\ldots\cos_{p-1}\cos u_p+
                  \tilde D_A^{(p)}\cos u_1\ldots\cos_{p-1}\sin u_p\quad\\[2mm]
              \quad  + \tilde D_A^{(p-1)}\cos u_1\ldots\sin_{p-1}+\ldots+
                  \tilde D_A^{(2)}\cos u_1\sin u_1+\tilde D_A^{(1)}\sin u_1+\tilde D_A\, ,
\end{array}
\end{equation}
where $D_A^{(p)},\ldots,D_A^{(0)}$, $D_A$,
$\tilde D_A^{(p)},\ldots,\tilde D_A^{(0)}$ and $\tilde D_A$ are real constants.
At this point let us make the following notations
$$
\begin{array}{lcl}
w_0 &=& \cos u_1\ldots\cos u_{p-1}\cos u_p\\[2mm]
w_p &=& \cos u_1\ldots\cos u_{p-1}\sin u_p\\[2mm]
w_{p-1} &=& \cos u_1\ldots\sin u_{p-1}\\
\ldots & \ldots & \ldots\ldots\ldots\ldots\ldots\\
w_2 &=& \cos u_1\sin u_2\\[2mm]
w_1 &=& \sin u_1.
\end{array}
$$
It follows that $\lambda_A$ and $\rho_A$ may be rewritten as
\begin{equation}
\label{lambda_rho_2}
\lambda_A(w)=D_A+\sum\limits_{a=0}^pD_A^{(a)}w_a\, ,\quad
\rho_A(w)=\tilde D_A+\sum\limits_{a=0}^p\tilde D_A^{(a)}w_a.
\end{equation}
Going back to \eqref{sol:DDp} we get, after a re-scaling with $a_1\neq0$
\begin{equation}
\label{xAyA}
\begin{array}{l}
x_A(s,t,w)=\(a_1s_1+\sum\limits_2^ha_js_j+\sum\limits_2^hb_jt_j\)\sum\limits_{a=0}^pD_A^{(a)}w_a\quad\\[2mm]
 \qquad  +   \(a_1t_1+\sum\limits_2^ha_jt_j+\sum\limits_2^hb_js_j\)\sum\limits_{a=0}^p\tilde D_A^{(a)}w_a+
        \sum\limits_{j=1}^h(\alpha_A^js_j+\beta_A^jt_j)\, , \\[3mm]
y_A(s,t,w)=\(a_1s_1+\sum\limits_2^ha_js_j+\sum\limits_2^hb_jt_j\)\sum\limits_{a=0}^p\tilde D_A^{(a)}w_a\quad\\[2mm]
 \qquad  +   \(a_1t_1+\sum\limits_2^ha_jt_j+\sum\limits_2^hb_js_j\)\sum\limits_{a=0}^p D_A^{(a)}w_a+
        \sum\limits_{j=1}^h(\tilde \alpha_A^js_j+\tilde\beta_A^jt_j).
\end{array}
\end{equation}
Let us choose the initial conditions
\begin{subequations}
\renewcommand{\theequation}{\theparentequation .\alph{equation}}
\label{init_cond}
\begin{align}
\label{IC1} & \Phi_*\partial_{s_i}(1,0,\ldots,0)=(0,\ldots,0,\stackrel{(i)}{1},0,\ldots,0,0,\ldots,0)\, ,\\
\label{IC2} & \Phi_*\partial_{t_i}(1,0,\ldots,0)=(0,\ldots,0,0,\ldots,\stackrel{(m+i)}{1},0,\ldots,0)\, ,\ i=1,\ldots,h\, ,\\
\label{IC3} & \Phi_*\partial_{u_b}(1,0,\ldots,0)=(0,\ldots,0,0,\ldots,\stackrel{(m+h+b)}{a_1 ,}0,\ldots,0)\, ,\ b=1,\ldots,p\,.
\end{align}
\end{subequations}
From \eqref{xAyA} and \eqref{IC3} and taking into account that
$$
\left.\frac{\partial w_a}{\partial u_b}\right|_{u=0}=\left\{
\begin{array}{l}
0, {\rm\ if\ } a=0\\
0, {\rm\ if\ } b\neq a, \ a\geq1\\
1, {\rm\ if\ } b=a\, ,
\end{array}\right.
$$
we obtain that
\begin{equation}
\label{22}
\begin{array}{l}
D_i^{(b)}=0,\ D_{h+a}^{(b)}=0,\ \tilde D_i^{(b)}=0,\ \tilde D_{h+a}^{(b)}=0, (a\neq b),\ \tilde D_{h+b}^{(b)}=1,\\[2mm]
\qquad\qquad\qquad  i=1,\ldots,h;\ a,b=1,\ldots,p.
\end{array}
\end{equation}
From \eqref{xAyA} and \eqref{IC1} we find
\begin{equation}
\label{23}
\begin{array}{l}
a_iD_j^{(0)}+b_i\tilde D_j^{(0)}+\alpha_j^i=\delta_{ij},\ a_iD_{h+a}^{(0)}+b_i\tilde D_{h+a}^{(0)}+\alpha_{h+a}^i=0\, ,\\[2mm]
a_i\tilde D_j^{(0)}+b_i D_j^{(0)}+\tilde\alpha_j^i=0,\ \ a_i\tilde D_{h+a}^{(0)}+b_i D_{h+a}^{(0)}+\tilde\alpha_{h+a}^i=0\, ,\\[2mm]
\qquad\qquad\qquad i,j=1,\ldots,h,\ a=1,\ldots,p,\ b_1=0.
\end{array}
\end{equation}
Finally, from \eqref{xAyA} and \eqref{IC2} we get
\begin{equation}
\label{24}
\begin{array}{l}
b_iD_j^{(0)}+a_i\tilde D_j^{(0)}+\beta_j^i=0,\ \ b_iD_{h+a}^{(0)}+a_i\tilde D_{h+a}^{(0)}+\beta_{h+a}^i=0\, ,\\[2mm]
b_i\tilde D_j^{(0)}+a_i D_j^{(0)}+\tilde\beta_j^i=\delta_{ij}, \ b_i\tilde D_{h+a}^{(0)}+a_i D_{h+a}^{(0)}+\tilde\beta_{h+a}^{i}=0\, ,\\[2mm]
\qquad\qquad\qquad i,j=1,\ldots,h,\ a=1,\ldots,p,\ b_1=0\, .
\end{array}
\end{equation}
Now, plugging \eqref{22}, \eqref{23} and \eqref{24} in \eqref{xAyA} we obtain
\begin{subequations}
\renewcommand{\theequation}{\theparentequation .\alph{equation}}
\label{xAyA1}
\begin{align}
\label{xD}
x_i(s,t,w)&=s_i+D_i^{(0)}(w_0-1)\(a_1s_1+\sum\limits_2^ha_js_j+\sum\limits_2^hb_jt_j\)\\
\nonumber    &  \quad+\tilde D_i^{(0)}(w_0-1)\(a_1t_1+\sum\limits_2^ha_jt_j+\sum\limits_2^hb_js_j\)\, ,
\end{align}
\begin{align} 
\label{xDp}
x_{h+a}(s,t,w) &= D_{h+a}^{(0)}(w_0-1)\(a_1s_1+\sum\limits_2^ha_js_j+\sum\limits_2^hb_jt_j\)\\
\nonumber    &  \ +\big[{w_a}+\tilde D_{h+a}^{(0)}(w_0-1)\big]\(a_1t_1+\sum\limits_2^ha_jt_j+\sum\limits_2^hb_js_j\)\, ,
\\
\label{yD}
y_i(s,t,w)&=t_i+ D_i^{(0)}(w_0-1)\(a_1t_1+\sum\limits_2^ha_jt_j+\sum\limits_2^hb_js_j\)\\
\nonumber    &  \quad+\tilde D_i^{(0)}(w_0-1)\(a_1s_1+\sum\limits_2^ha_js_j+\sum\limits_2^hb_jt_j\)\, ,
\\
\label{yDp}
y_{h+a}(s,t,w)&=D_{h+a}^{(0)}(w_0-1)\(a_1t_1+\sum\limits_2^ha_jt_j+\sum\limits_2^hb_js_j\)\\
\nonumber    &  \ +\big[{w_a}+\tilde D_{h+a}^{(0)}(w_0-1)\big]\(a_1s_1+\sum\limits_2^ha_js_j+\sum\limits_2^hb_jt_j\).
\end{align}
\end{subequations}
Since $\Phi$ is an isometric immersion we have $\tilde g(\Phi_* U,\Phi_* V)=g(U,V)$ for every $U$ and $V$ tangent to $M$.
From $\tilde g(\Phi_* \partial s_1,\Phi_* \partial s_1)=-1$ and \eqref{xAyA1} we get
$$
(w_0-1)\langle D^{(0)},D^{(0)}\rangle+2\sum\limits_{a=1}^pw_a\tilde D_{h+a}^{(0)}-\frac2{a_1}~D_1^{(0)}-(w_0+1)=0
$$
for all $w\in{\mathbb{S}}^p$, where
$$D^{(0)}=\(D_1^{(0)},\ldots,D_h^{(0)},D_{h+1}^{(0)},\ldots,D_{2h}^{(0)},
\tilde D_1^{(0)},\ldots,\tilde D_h^{(0)},\tilde D_{h+1}^{(0)},\ldots,\tilde D_{2h}^{(0)}\).
$$
Therefore
\begin{equation}
\label{isom11}
D_1^{(0)}=-{a_1}\, ,\ \tilde D_{h+a}^{(0)}=0, \forall a=1,\ldots,p,\ \langle D^{(0)},D^{(0)}\rangle=1\, .
\end{equation}
From $\tilde g(\Phi_*\partial s_1,\Phi_*\partial s_j)=0$ and $\tilde g(\Phi_*\partial s_1,\Phi_*\partial t_j)=0$, $(j\geq2)$,
together with \eqref{xAyA1} and \eqref{isom11} it follows
\begin{equation}
\label{isom1j}
D_j^{(0)}=-{a_j}-\frac{b_j}{a_1}~\tilde D_1^{(0)},\quad
\tilde D_j^{(0)}={b_j}+\frac{a_j}{a_1}~\tilde D_1^{(0)},\ \forall j\geq 2.
\end{equation}
Finally, from $\tilde g (\Phi_*\partial s_1,\Phi_*\partial u_b)=0$, \eqref{xAyA1} and \eqref{isom11} we get
$\tilde D_1^{(0)}=0$. Hence from \eqref{isom1j} one obtains
$D_j^{(0)}=-a_j$ and $\tilde D_j^{(0)}=b_j$, for all $j=1,\ldots,h$ (recall $b_1=0$),
which combined with $\langle D^{(0)},D^{(0)}\rangle=1$ yield $D_{h+a}^{(0)}=0.$

We conclude from \eqref{xAyA1} the following
\begin{equation}
\begin{array}{l}
x_i(s,t,w)=s_i-a_i(w_0-1)\(a_1s_1+\sum\limits_2^ha_js_j+\sum\limits_2^hb_jt_j\)\\
\qquad\qquad +b_i(w_0-1)\(a_1t_1+\sum\limits_2^ha_jt_j+\sum\limits_2^hb_js_j\),\\[2mm]
x_{h+a}(s,t,w)=w_a\(a_1t_1+\sum\limits_2^ha_jt_j+\sum\limits_2^hb_js_j\),
\end{array}
\end{equation}

\begin{equation}\notag
\begin{aligned}
&y_i(s,t,w)=t_i-a_i(w_0-1)\(a_1t_1+\sum\limits_2^ha_jt_j+\sum\limits_2^hb_js_j\)\\
&\qquad\qquad +b_i(w_0-1)\(a_1s_1+\sum\limits_2^ha_js_j+\sum\limits_2^hb_jt_j\),\\&
y_{h+a}(s,t,w)=w_a\(a_1s_1+\sum\limits_2^ha_js_j+\sum\limits_2^hb_jt_j\).
\end{aligned}
\end{equation}
Computing now $x_i+\j y_i$ and $x_{h+a}+\j y_{h+a}$ one gets \eqref{case1}.

\smallskip

Let consider the second situation when $\Na$ is the hyperbolic space ${\mathbb{H}}^p$.
On ${\mathbb{H}}^p$ consider coordinates $u=(u_1,u_2,\ldots,u_p)$ such that the metric $g_\perp$
is expressed by
\begin{equation}
\label{metricHp}
g_\perp=du_1^2+\sinh^2u_1\left(du_2^2+\cos^2u_2 du_3^2+\ldots+
          \cos^2u_2\ldots\cos^2u_{p-1}du_p^2\right),
\end{equation}
and the warped metric on $M$ is given by
$g=g_\top(s,t)+f^2(s,t)g_\perp(u). $
Then the Levi Civita connection $\nabla$ of $g$ satisfies
\begin{subequations}
\renewcommand{\theequation}{\theparentequation .\alph{equation}}
\label{eq:LC:Hp}
\begin{align}
\label{H4} & \nabla_{\dsi}\dsj=0\, ,\ \nabla_{\dsi}\dtj=0\, ,\ \nabla_{\dti}\dtj=0,\\
\label{H5} &\nabla_{\dsi}\dua=\frac{f_{s_i}}f~\dua\, ,\ \nabla_{\dti}\dua=\frac{f_{t_i}}f~\dua,\\
\label{H6} &\nabla_{\partial{u_1}}\dub=\coth u_1\dub\quad(1<b),\\
\label{H7} &\nabla_{\dua}\dub=-\tan u_a\dub\quad(1<a<b),
\\ \label{H8} &\nabla_{\partial_{u_1}}\partial_{u_1}=\sum\limits_{i=1}^h\big(ff_{s_i}\dsi-ff_{t_i}\dti\big),
\\\label{H9} & \nabla_{\dua}\dua=\sinh^2u_1\prod\limits_{b=2}^{a-1}\cos^2u_b\sum\limits_{i=1}^h\big(ff_{s_i}\dsi-ff_{t_i}\dti\big)\\
\nonumber & \qquad\qquad -\sinh u_1\cosh u_1\prod\limits_{b=2}^{a-1}\cos^2u_b~\partial_{u_1} \\
\nonumber & \qquad + \sum\limits_{b=1}^{a-1}\big(\sin u_b\cos u_b\cos^2u_{b+1}\ldots\cos^2u_{a-1}\big)\dub,\;\;  (1<a)
\end{align}
\end{subequations}
for any $ i,j =1,\ldots,h$ and $a, b=1,\ldots,p$.

In the following we will proceed in the same way as in previous case.
Since some computations are very similar we will skip them, and we will focus only on the major differences between the two cases.

The function $\psi$ is obtained from Proposition~\ref{PDEsyst} (case 1 in the proof):
$$
\psi=\frac12\ln\left[\langle \bar v,z\rangle^2-\langle \j \bar v,z\rangle^2\right],
$$
where $v=(a_1,a_2,\ldots,a_h,0,b_2,\ldots,b_h)$, with $a_1>0$ is a constant vector.

Applying  Gauss' formula $\widetilde \nabla_{\Phi_*U}\Phi_*V=\Phi_*\nabla_UV+\sigma(U,V)$
for $U,V\in\D$, respectively for $U\in\D$ and $V\in\Dp$ we may write \eqref{sol:DDp}.
Using  Gauss' formula for $U=V=\partial_{u_1}$, we find
$$
\begin{array}{l}
\dfrac{\partial\lambda_A}{\partial u_1^2}+\langle v,v\rangle\lambda_A-D_A=0\quad: \quad D_A=\sum a_jc^j_A-\sum b_j\tilde c^j_A\\[2mm]
\dfrac{\partial\rho_A}{\partial u_1^2}+\langle v,v\rangle\rho_A-\tilde D_A=0\quad: \quad D_A=\sum b_jc^j_A-\sum a_j\tilde c^j_A.
\end{array}
$$
Here $\langle v,v\rangle={||\nabla f||}_2=-1$ and consequently
\begin{equation}
\label{H13}
\begin{array}{l}
\lambda_A(u)=\cosh u_1 D_A^{(0)}(u_2,\ldots,u_p)+\sinh u_1 \Theta_A^{(0)}(u_2,\ldots,u_p)-D_A\, ,\\[2mm]
\rho_A(u)=\cosh u_1 \tilde D_A^{(0)}(u_2,\ldots,u_p)+\sinh u_1 \tilde \Theta_A^{(0)}(u_2,\ldots,u_p)-\tilde D_A.
\end{array}
\end{equation}
Taking $U=\partial_{u_1}$ and $V=\dub$, ($b>1$) we find that $D_A^{(0)}$ and $\tilde D_A^{(0)}$ are constants.
Next, applying the Gauss formula for $U=V=\partial_{u_2}$ and respectively for $U=\partial_{u_2}$ and $V=\dub$, ($b>2$)
we get
$$
\begin{array}{l}
\Theta_A^{(0)}=\cos u_2 \Theta_A^{(1)}(u_3,\ldots,u_p)+ D_A^{(1)}\sin u_2\, ,\\[2mm]
\tilde\Theta_A^{(0)}=\cos u_2 \tilde \Theta_A^{(1)}(u_3,\ldots,u_p)+ \tilde D_A^{(1)}\sin u_2,\quad D_A^{(1)},\tilde D_A^{(1)}\in{\mathbb{R}}.
\end{array}
$$
Continuing the procedure sufficiently many times we finally get
$$
\begin{array}{l}
\lambda_A=-D_A+D_A^{(0)}\cosh u_1+D_A^{(1)}\sinh u_1\cos u_2+D_A^{(2)}\sinh u_1\cos u_2\sin u_3+\cdots\\
 \qquad +D_A^{p-1)}\sinh u_1\cos u_2\cdots\cos u_{p-1}\sin u_p+D_A^{(p)}\sinh u_1\cos u_2\cdots \cos u_p\, , \\[2mm]
\rho_A=-\tilde D_A+\tilde D_A^{(0)}\cosh u_1+\tilde D_A^{(1)}\sinh u_1\cos u_2+\tilde D_A^{(2)}\sinh u_1\cos u_2\sin u_3+\cdots\\
 \qquad +\tilde D_A^{p-1)}\sinh u_1\cos u_2\cdots\cos u_{p-1}\sin u_p+\tilde D_A^{(p)}\sinh u_1\cos u_2\cdots \cos u_p.
\end{array}
$$
Considering the hyperbolic space ${\mathbb{H}}^p$ embedded in ${\mathbb{R}}^{p+1}_1$ with coordinates
\begin{equation}
\label{eq:wHp}
\begin{array}{l}
w_0=\cosh u_1\\
w_1=\sinh u_1\sin u_2\\
w_2=\sinh u_1\cos u_2\sin u_3\\
\ldots \ldots \ldots \\
w_{p-1}=\sinh u_1\cos u_2\ldots\cos u_{p-1}\sin u_p\\
w_p=\sinh u_1\cos u_2\ldots\cos u_{p-1}\cos u_p\, ,
\end{array}
\end{equation}
we may express $\lambda_A$ and $\rho_A$ in terms of $w=(w_0,w_1,\ldots,w_p)$:
\begin{equation}
\label{H17}
\begin{array}{l}
\lambda_A=-D_A+D_A^{(0)}w_0+D_A^{(1)}w_1+\ldots+D_A^{(p)}w_p\, ,\\[2mm]
\rho_A=-\tilde D_A+\tilde D_A^{(0)}w_0+\tilde D_A^{(1)}w_1+\ldots+\tilde D_A^{(p)}w_p\, .
\end{array}
\end{equation}
After a rescaling with the factor $a_1\neq0$ we may write
$$
\begin{array}{l}
x_A(s,t,w)=\(a_1s_1+\sum\limits_2^ha_js_j+\sum\limits_2^hb_jt_j\)\sum\limits_{a=0}^pD_A^{(a)}w_a\quad\\[2mm]
 \qquad  +   \(a_1t_1+\sum\limits_2^ha_jt_j+\sum\limits_2^hb_js_j\)\sum\limits_{a=0}^p\tilde D_A^{(a)}w_a+
        \sum\limits_{j=1}^h(\alpha_A^js_j+\beta_A^jt_j)\, ,\\[3mm]
y_A(s,t,w)=\(a_1s_1+\sum\limits_2^ha_js_j+\sum\limits_2^hb_jt_j\)\sum\limits_{a=0}^p\tilde D_A^{(a)}w_a\quad\\[2mm]
 \qquad  +   \(a_1t_1+\sum\limits_2^ha_jt_j+\sum\limits_2^hb_js_j\)\sum\limits_{a=0}^p D_A^{(a)}w_a+
        \sum\limits_{j=1}^h(\tilde \alpha_A^js_j+\tilde\beta_A^jt_j)
\end{array}
$$
which is similar to \eqref{xAyA}. From now on we will put
\begin{equation}
\label{SandT}
S=a_1s_1+\sum\limits_2^ha_js_j+\sum\limits_2^hb_jt_j
\quad {\rm and} \quad
T=a_1t_1+\sum\limits_2^ha_jt_j+\sum\limits_2^hb_js_j.
\end{equation}

Choose the initial point $s_{\rm{init}}(1,0,\ldots,0)$, $t_{\rm{init}}=(0,0,\ldots,0)$, $u_{\rm{init}}=(\omega,0,\ldots,0)$
with $\omega\neq0$ and the initial conditions
$$
\begin{array}{l}
 \Phi_*\partial_{s_i}(1,0,\cdots,0,\omega,0,\cdots,0)=(0,\cdots,0,\stackrel{(i)}{1},0,\cdots,0,0,\cdots,0)\,,\\
 \Phi_*\partial_{t_i}(1,0,\cdots,0,\omega,0,\cdots,0)=(0,\cdots,0,0,\cdots,\stackrel{(m+i)}{1},0,\cdots,0)\, ,\ i=1,\cdots,h\, ,\\
 \Phi_*\partial_{u_1}(1,0,\cdots,0,\omega,0,\cdots,0)=(0,\cdots,0,0,\cdots,\stackrel{(m+h+1)}{a_1 ,}0,\cdots,0)\, , \\
 \Phi_*\partial_{u_b}(1,0,\cdots,0,\omega,0,\cdots,0)=(0,\cdots,0,0,\cdots,\stackrel{(m+h+b)}{a_1\sinh\omega ,}0,\cdots,0),\ b=2,\cdots,p\,.
\end{array}
$$
A straightforward computations, similar to previous case, yield
$$
\begin{array}{l}
x_i(s,t,w)=s_i+a_i\big(W_0-1\big){S}-   b_i\big(W_0-1\big){T}\, ,\\[2mm]
x_{h+1}(s,t,w)=W_p{T},\;\;
x_{h+a}(s,t,w)=w_{a-1}{T}\ , \;\; a=2,\ldots,p\, ,\\[2mm]
y_i(s,t,w)=t_i+a_i\big(W_0-1\big){T}-
    b_i\big(W_0-1\big){S}\, ,\\[2mm]
y_{h+1}(s,t,w)=W_p{S}\, ,\;\; 
y_{h+a}(s,t,w)=w_{a-1}{S}\ , \;\; a=2,\ldots,p\, ,
\end{array}
$$
where $W_0=w_0\cosh\omega-w_p\sinh\omega$ and $W_p=-w_0\sinh\omega+w_p\cosh\omega$.
Moreover, since $W_0^2-W_p^2=w_0^2-w_p^2$, it follows $(W_0,w_1,\ldots,w_{p-1},W_p)\in{\mathbb{H}}^p$ and after a re-notation
we write
$$
\begin{array}{l}
x_i(s,t,w)=s_i+a_i\big(w_0-1\big){S}-    b_i\big(w_0-1\big){T}\, ,\\[2mm]
x_{h+a}(s,t,w)=w_{a}{T}\, , \quad a=1,\ldots,p\, ,\\[2mm]
y_i(s,t,w)=t_i+a_i\big(w_0-1\big){T}-
    b_i\big(w_0-1\big){S}\, ,\\[2mm]
y_{h+a}(s,t,w)=w_{a}{S}\, , \quad a=1,\ldots,p\, ,
\end{array}
$$
where $(w_0,w_1,\ldots,w_p)\in{\mathbb{H}}^p$.
Computing now $x_i+\j y_i$ and $x_{h+a}+\j y_{h+a}$ gives \eqref{case2}.

\smallskip

Let consider the third situation when $\Na$ is the flat space ${\mathbb{E}}^p$, on which
we take coordinates $u=(u_1,u_2,\ldots,u_p)$ such that the metric $g_\perp$
is expressed by
\begin{equation}
\label{metricEp}
g_\perp=du_1^2+\ldots+du_p^2.
\end{equation}
Then the warped metric on $M$ is given by
$g=g_\top(s,t)+f^2(s,t)g_\perp(u).$
Then the Levi Civita connection $\nabla$ of $g$ satisfies
\begin{subequations}
\renewcommand{\theequation}{\theparentequation .\alph{equation}}
\label{eq:LC:Ep}
\begin{align}
\label{F4} & \nabla_{\dsi}\dsj=0\ ,\ \nabla_{\dsi}\dtj=0\ ,\ \nabla_{\dti}\dtj=0\, ,\\
\label{F5} &\nabla_{\dsi}\dua=\frac{f_{s_i}}f~\dua\, ,\ \nabla_{\dti}\dua=\frac{f_{t_i}}f~\dua,\\
\label{F7} &\nabla_{\dua}\dub=0\, ,\ (a\neq b)\,,\\
\label{F8} &\nabla_{\partial_{u_a}}\partial_{u_a}=\sum\limits_{i=1}^h\big(ff_{s_i}\dsi-ff_{t_i}\dti\big)\, ,
\end{align}
\end{subequations}
for any $ i,j =1,\ldots,h$ and $a, b=1,\ldots,p$.

In the following we will proceed in the same way as in previous cases.
Again, we skip most computations, emphasizing only the major differences appearing in this situation.
The function $\psi$ is obtained from Proposition~\ref{PDEsyst} (case 1 in the proof):
$$
\psi=\frac12\ln\left[\langle \bar v,z\rangle^2-\langle \j \bar v,z\rangle^2\right],
$$
where $v=(a_1,\ldots,a_h,0,t_2,\ldots,t_h)$, $a_1>0,$ is a constant vector.
Applying  Gauss' formula $\widetilde \nabla_{\Phi_*U}\Phi_*V=\Phi_*\nabla_UV+\sigma(U,V)$
for $U,V\in\D$, respectively for $U\in\D$ and $V\in\Dp$ we may write \eqref{sol:DDp}.
Using Gauss' formula for $U=V=\partial_{u_1}$, we find
$$
\begin{array}{l}
\dfrac{\partial\lambda_A}{\partial u_1^2}+\langle v,v\rangle\lambda_A-D_A=0\quad: \quad D_A=\sum a_jc^j_A-\sum b_j\tilde c^j_A\\[2mm]
\dfrac{\partial\rho_A}{\partial u_1^2}+\langle v,v\rangle\rho_A-\tilde D_A=0\quad: \quad D_A=\sum b_jc^j_A-\sum a_j\tilde c^j_A\, .
\end{array}
$$
Here $\langle v,v\rangle={||\nabla f||}_2=0$.
Taking $U=\partial_{u_1}$ and $V=\dub$ ($b>1$) we find that $\frac{\partial^2\lambda_A}{\partial u_1\partial u_b}=0$
and $\frac{\partial^2\rho_A}{\partial u_1\partial u_b}=0$.
As consequence,
$$
\begin{array}{l}
\lambda_A(u)=\frac{D_A}2~u_1^2+D_A^{(1)}u_1+\Theta_A^{(1)}(u_2,\ldots,u_p)\, ,\\[2mm]
\rho_A(u)=\frac{\tilde D_A}2~u_1^2+\tilde D_A^{(1)}u_1+\tilde \Theta_A^{(1)}(u_2,\ldots,u_p)\, ,
\end{array}
$$
where $D_A^{(1)},\tilde D_A^{(1)}$ are constants.
Continuing the computations in the same manner it turns that
\begin{equation}
\label{F13}
\begin{array}{l}
\lambda_A(u)=\frac{D_A}2~\sum\limits_{a=1}^pu_a^2+\sum\limits_{a=1}^pD_A^{(a)}u_a+D_A^{(0)},\\[2mm]
\rho_A(u)=\frac{\tilde D_A}2~\sum\limits_{a=1}^pu_a^2+\sum\limits_{a=1}^p\tilde D_A^{(a)}u_a+\tilde D_A^{(0)},
\end{array}
\end{equation}
where $D_A^{(0)},\tilde D_A^{(0)}$ and $D_A^{(a)},\tilde D_A^{(a)}$, $a=1,\ldots,p$ are constants.
Choosing suitable initial conditions and taking into account the property of $\Phi$ to be isometric immersion, 
straightforward computations yield
\begin{equation}
\label{F18}
\begin{array}{l}
x_i=s_i+\frac12\big(a_i{S}   -b_i{T}\big){\sum\limits_{1}^pu_a^2}\, ,\;\; 
x_{h+b}=u_b{T}\, ,\\[2mm]
y_i=t_i+\frac12\big(a_i{T}-   b_i{S}\big){\sum\limits_{1}^pu_a^2}\, ,\;\; 
y_{h+b}=u_b{S}\, ,
\end{array}
\end{equation}
where $S$ and $T$ are as in \eqref{SandT}.
Computing now $x_i+\j y_i$ and $x_{h+b}+\j y_{h+b}$ one gets \eqref{case3}.
In the end, consider $\Na^0=\{(s_0,t_0)\}\times{\mathbb{E}}^p$, where $(s_0,t_0)$ is a fixed
point in ${\mathbb{E}}^{2h}_h$.
If $\sigma_\perp^0$ is the second fundamental form of $\Na^0$ in ${\mathbb{E}}^{2m}_m$,
we find ${||\sigma_\perp^0(\dua,\dua)||}_2=0.$ So, the mean curvature vector of $\Na^0$ is a light-like vector, so
it is nowhere zero.

If $h=1$, then $v=(a_1,0)$. Thus $||v||_2<0$. Hence, $N_\perp$ is an open part of  the hyperbolic space $\mathbb H^p$. So, we obtain item {\bf 2}.

\smallskip

Let us now consider the case $p=1$.
In this case $\Na$ is a curve, which can be supposed to be parameterized by the arc-length $u$. Hence its metric
is $g_\perp=du^2$.
We can make the same computations as in previous case such that \eqref{sol:DDp} holds. Yet, a first difference
appear: we are not able to say something about the value of ${||\nabla f||}_2=-\sum\limits_{i=1}^ha_i^2+\sum\limits_{i=1}^hb_i^2$.

Using as usual  Gauss' formula (for $U=V=\dua$) one gets
$$
\frac{\partial^2\lambda_A}{\partial u^2}=\langle v,v\rangle\lambda_A+D_A\, ,
\quad
\frac{\partial^2\rho_A}{\partial u^2}=\langle v,v\rangle\rho_A+\tilde D_A\, ,
$$
where $D_A, \tilde D_A\in{\mathbb{R}}$.
Since $\langle v,v\rangle=-\sum\limits_{i=1}^ha_i^2+\sum\limits_{i=1}^hb_i^2$ is an arbitrary constant, we have to distinguish three different cases:
{\bf Case (i)} $\langle v,v\rangle=-r^2$,
{\bf Case (ii)} $\langle v,v\rangle=r^2$ and
{\bf Case (iii)} $\langle v,v\rangle=0$ ($r>0$).

Solving the ordinary differential equations and doing the computations in the same manner as in the case when $p>1$, and after a re-scaling of the
vector $v$, we obtain the first three cases stated in the theorem.

\smallskip

At this point we recall that the PDE system in Proposition~\ref{PDEsyst} has also other solutions.
When Case 2a from the proof is considered, doing similar computations we easily get item {\bf 4} of the theorem.

Much more interesting is to consider Case 2b in the proof of Proposition~\ref{PDEsyst}.
We have to examine again the three situations, namely when  $\Na$ is ${\mathbb{S}}^p$,
${\mathbb{H}}^p$ or ${\mathbb{E}}^p$. In the following we give only few details for the case $M={\mathbb{E}}^{2h}_h\times_f{\mathbb{S}}^p$,
the other two being very similar. Here the warping function is $f=\sqrt{AB}$, where
$$
A=\sum\limits_{k=1}^ha_k(s_k+\epsilon t_k)\ ,\quad
B=\sum\limits_{k=1}^hb_k(s_k-\epsilon t_k)\, ,
$$
$\epsilon=\pm1$, $a_1=0$, $b_1=1$, $a_2\neq0$. Moreover, by Proposition~\ref{prop_v} we get
$\sum\limits_{k=1}^ha_kb_k=-1$.

Direct computations, analogue to those done in the first part of the proof, yield
\begin{equation}
\label{F18}
\begin{array}{l}
x_i=s_i+\dfrac{w_0-1}2\big(b_i{A}+a_i{B}\big)\, ,\;\;
x_{h+b}=\dfrac{u_b}2\big(A-B)\, ,\\[2mm]
y_i=t_i+\epsilon \dfrac{w_0-1}2\big(b_i{A}-a_i{B}\big)\, ,\;\; 
x_{h+b}=\epsilon\dfrac{u_b}2\big(A+B)\, ,
\end{array}
\end{equation}
where $(w_0,w_1,\ldots,w_p)\in{\mathbb{S}}^p$.
Put $v_k=\frac\epsilon2(a_k+b_k)+\frac 12\j(a_k-b_k)$. We have $\langle v,v\rangle=1$,
where $v=(v_1,\ldots,v_p)$.
Computing $x_i+\j y_i$ and $x_{h+b}+\j y_{h+b}$ we obtain
\eqref{case1}. Moreover, the warping function could be written as
$f=\sqrt{\langle \bar v,z\rangle^2-\langle \j \bar v,z\rangle^2}$. So, we obtain again item {\bf 1} of the theorem.

The converse follows  from direct computations.
\end{proof}

\begin{remark}
\rm In the case 3 of previous proof, if we choose
$(s_0,t_0)=(1,0,\ldots,0)$, and $v=(1,0,\ldots,0,\sqrt{3}+2\j)$,
we obtain the ``initia'' leaf $\Na^0$ given by
$$
\begin{array}{l}
\Phi(1,0,u)=\(1+\frac12\sum u_a^2,0,\ldots,0,\stackrel{(h)}{\frac{\sqrt{3}}2\sum u_a^2},0,\ldots,0,\stackrel{(m+h)}{-\sum u_a^2},u_1,\ldots,u_p)\, .
\end{array}
$$
After a translation along $x_1$ axis, followed by a rotation in the 2-plane $(x_1,x_h)$ of a suitable angle,
we obtain
$$
\begin{array}{l}
\Phi(1,0,u)=\(0,\ldots,0,\stackrel{(h)}{-\sum u_a^2},0,\ldots,0,\stackrel{(m+h)}{-\sum u_a^2},u_1,\ldots,u_p).
\end{array}
$$
which represents the submanifold given in \cite[Proposition 3.6]{cm:Chen11}  up to reordering of coordinates.
\end{remark}

\begin{remark} {\rm By applying the same method we may also classify all time-like $\p R$-warped products   $\Ni\times_f\Na$  in the  para-K\"ahler $(h+p)$-plane $\p^{h+p}$ satisfying $h=\frac{1}{2}\dim \Ni$, $p=\dim \Na$ and $S_\sigma= 2p{||\nabla\ln f||}_2$.}
\end{remark}

{\bf Acknowledgement.} The second author is supported by Fulbright Grant no. 498/2010 as a Fulbright Senior Researcher at Michigan State University, U.S.A.


\end{document}